\documentclass[a4paper,12pt]{amsart}

\usepackage[T1]{fontenc}
\usepackage{inputenc}
\usepackage[english]{babel}
\usepackage{amsmath}
\usepackage{latexsym}
\usepackage{marvosym}
\usepackage{amsfonts}
\usepackage{amsthm}
\usepackage{float}
\usepackage{color}
\usepackage{epsfig}
\usepackage{graphicx}

\usepackage{enumerate}

\usepackage{stmaryrd}

\usepackage{tikz}
\usetikzlibrary{trees}
\usepackage{verbatim}
\usetikzlibrary{shapes,arrows}
\usepackage{graphicx}

\newcommand\N{\mathbb{N}}

\newcommand\E{\mathbb{E}}
\newcommand\F{\mathbb{F}}
\newcommand\R{\mathbb{R}}
\newcommand\Pro{\mathbb{P}}
\newcommand\T{\mathbb{T}}
\newcommand\G{\mathbb{G}}

\newcommand{\veps}{\varepsilon}

\newcommand\calF{\mathcal{F}}
\newcommand\calG{\mathcal{G}}
\newcommand\calB{\mathcal{B}}
\newcommand\calO{\mathcal{O}}
\newcommand\calN{\mathcal{N}}
\newcommand\calC{\mathcal{C}}
\newcommand\calV{\mathcal{V}}
\newcommand\calA{\mathcal{A}}
\newcommand\calW{\mathcal{W}}
\newcommand\calR{\mathcal{R}}
\newcommand\calT{\mathcal{T}}
\newcommand\calH{\mathcal{H}}
\newcommand\calK{\mathcal{K}}

\newcommand\wh{\widehat}

\newcommand\Var{\text{Var}}

\def\build#1_#2^#3{\mathrel{\mathop{\kern 0pt#1}\limits_{#2}^{#3}}}
\def\liml{\build{\longrightarrow}_{}^{{\mbox{$\mathcal L$}}}}

\def\H#1{\textup{\textbf{(H.\ref{#1})}}}

\numberwithin{equation}{section}
\theoremstyle{plain}
\newtheorem{TD}{Theorem-Definition}[section]

\newtheorem{Prop}[TD]{Proposition}
\newtheorem{Theo}[TD]{Theorem}

\newtheorem{Lem}[TD]{Lemma}
\newtheorem{Rem}[TD]{Remark}

\begin{document}

\title[Asymptotic analysis for BINAR processes]
{Limit theorems for bifurcating integer-valued autoregressive processes}
\author{Vassili Blandin}
\dedicatory{\normalsize Universit\'e Bordeaux 1}



\begin{abstract}
We study the asymptotic behavior of the weighted least squares estimators of the unknown parameters of bifurcating integer-valued autoregressive processes. Under suitable assumptions on the immigration, we establish the almost sure convergence of our estimators, together with the quadratic strong law and central limit theorems. All our investigation relies on asymptotic results for vector-valued martingales.
\end{abstract}

\maketitle


\section{Introduction}


Bifurcating integer-valued autoregressive (BINAR) processes are an adaptation of integer-valued autoregressive (INAR) processes to binary
tree structured data. It can also be seen as the combination of INAR processes and bifurcating autoregressive (BAR) processes. BAR processes have been first introduced by Cowan and Staudte \cite{CowanStaudte} while INAR processes have been first investigated by Al-Osh and Alzaid \cite{AlOshAlzaid, AlzaidAlOsh} and McKenzie \cite{McKenzie}. BINAR processes take into account both inherited and environmental effects to explain the evolution of the integer-valued characteristic under study. We can easily see cell division as an example of binary tree structured, the integer-valued characteristic could then be, as an example, the number of parasites in a cell.\\

More precisely, the first-order BINAR process is defined as follows. The initial cell is labelled $1$ and the offspring of the cell labelled $n$ are labelled $2n$ and $2n+1$. Denote by $X_n$ the integer-valued characteristic of individual $n$. Then, the first-order BINAR process is given, for all $n\geq1$, by
\begin{equation*}
\begin{cases}
X_{2n} &= a \circ X_n + \veps_{2n}\\
X_{2n+1} &= b \circ X_n + \veps_{2n+1}
\end{cases}
\end{equation*}
where the thinning operator $\circ$ is defined in \eqref{thinning}. The immigration sequence $(\veps_{2n},\veps_{2n+1})_{n\geq1}$ represents the environmental effect, while the thinning operator represents the inherited effect. The example of the cell division incites us to suppose that $\veps_{2n}$ ans $\veps_{2n+1}$ are correlated since the environmental effect on two sister cells can reasonably be seen as correlated.\\

The purpose of this paper is to study the asymptotic behavior of the weighted least squares (WLS) estimators of first-order BINAR process via a martingale approach. The martingale approach has been first proposed by Bercu et al.~\cite{BercuBDSAGP} and de Saporta et al.~\cite{BDSAGPMarsalle} for BAR processes. We also refer to Wei and Winnicki \cite{WeiWinnicki} and Winnicki \cite{Winnicki} for the WLS estimation of parameters associated to branching processes. We shall make use of the strong law of large numbers \cite{Duflo} as well as the central limit theorem \cite{Duflo,HallHeyde} for martingales, in order to investigate the asymptotic behavior of the WLS estimators, as previously done by Basawa and Zhou \cite{BasawaZhou,ZhouBasawa,ZhouBasawa2}.\\

Several points of view appeared for both BAR and INAR processes and we tried to make a link between those approaches. On the one hand, for the BAR side of the BINAR process, we had a look to classical BAR studies as done by Huggins and Basawa \cite{HugginsBasawa99,HugginsBasawa2000} and Huggins ans Staudte \cite{HugginsStaudte} who studied the evolution of cell diameters and lifetimes, but also to bifurcating Markov chains models introduced by Guyon \cite{Guyon} and used in Delmas and Marsalle \cite{DelmasMarsalle}. However, we did not put aside the analogy with the Galton-Watson processes as studied in Delmas and Marsalle \cite{DelmasMarsalle} and Heyde and Seneta \cite{HeydeSeneta}. On the other hand, concerning the INAR side of the BINAR process, we used the classical INAR definition but also had a look to Bansaye \cite{Bansaye} who studied an integer-valued process on a binary tree without using an INAR model, and also Kachour and Yao \cite{KachourYao} who decided to study an integer-valued autoregressive process by a rounding approach instead of the classical INAR one.\\

The paper is organised as follows. Section 2 is devoted to the presentation of the first-order BINAR process while Section 3 deals with the WLS estimators of the unknown parameters. Section 4 allows us to detail our approach based on martingales. Section 5 gathers the main results about the asymptotic properties of the WLS estimators. More precisely, we will propose the almost sure convergence, the quadratic strong law and the central limit theorem for our estimates. The rest of the paper is devoted to the proofs of our main results.


\section{Bifurcating integer-valued autoregressive processes}\label{section2}


Consider the first-order BINAR process given, for all $n\geq1$, by
\begin{equation}
\label{orisyst}
\begin{cases}
X_{2n} &= a \circ X_n + \veps_{2n}\\
X_{2n+1} &= b \circ X_n + \veps_{2n+1}
\end{cases}
\end{equation}
where
the initial integer-valued state $X_1$ is the ancestor of the process and $(\veps_{2n},\veps_{2n+1})$ represents the immigration which takes 
nonnegative integer values. In all the sequel, we shall assume that $\E[X_1^8]<\infty$. Moreover,
\begin{equation}\label{thinning}
\displaystyle{a \circ X_{n} = \sum_{i=1}^{X_n} Y_{n,i}} \hspace{15pt} \text{ and } \hspace{15pt} \displaystyle{b \circ X_{n} = \sum_{i=1}^{X_n} Z_{n,i}}
\end{equation}
where $(Y_{n,i})_{n,i \geq1}$ and $(Z_{n,i})_{n,i \geq1}$ are two independent sequences of i.i.d., nonnegative integer-valued random variables with means $a$ and $b$ and positive variances $\sigma_a^2$ and $\sigma_b^2$ respectively. Moreover, $\mu_a^4$, $\mu_b^4$ and $\tau_a^6$, $\tau_b^6$ are the fourth-order and the sixth-order centered moments of $(Y_{n,i})$ and $(Z_{n,i})$, respectively, and $(Y_{n,i})$ and $(Z_{n,i})$ admit eighth-order moments. We also assume that the two offspring sequences $(Y_{n,i})$ and $(Z_{n,i})$ are independent 
of the immigration $(\veps_{2n},\veps_{2n+1})$. In addition, as in the literature concerning BAR processes, we shall assume that 
$$0<\max(a,b)<1.$$

\noindent One can see this BINAR process as a first-order integer-valued autoregressive process on a binary tree, where each node 
represents an individual, node 1 being the original ancestor. For all $n\geq1$, denote the $n$-th generation by
$$\G_n = \{2^n, 2^{n+1}, \hdots, 2^{n+1}-1\}.$$

\noindent In particular, $\G_0 = \{1\}$ is the initial generation and $\G_1 = \{2,3\}$ is the first generation of offspring 
from the first ancestor. Let $\G_{r_n}$ be the generation of individual $n$, which means that $r_n=[\log_2(n)]$. Recall that the two 
offspring of individual $n$ are labelled $2n$ and $2n+1$, or conversely, the mother of individual $n$ is $[n/2]$ where $[x]$ stands for 
the largest integer less than or equal to $x$. Finally 
denote by
$$\T_n = \bigcup_{k=0}^n \G_n$$

\noindent the sub-tree of all individuals from the original individual up to the $n$-th generation. On can observe that the cardinality $|\G_n|$ 
of $\G_n$ is $2^n$ while that of $\T_n$ is $|\T_n| = 2^{n+1}-1$.\\

\begin{figure}[h!]
\hspace{-2cm}
\includegraphics[scale=0.7]{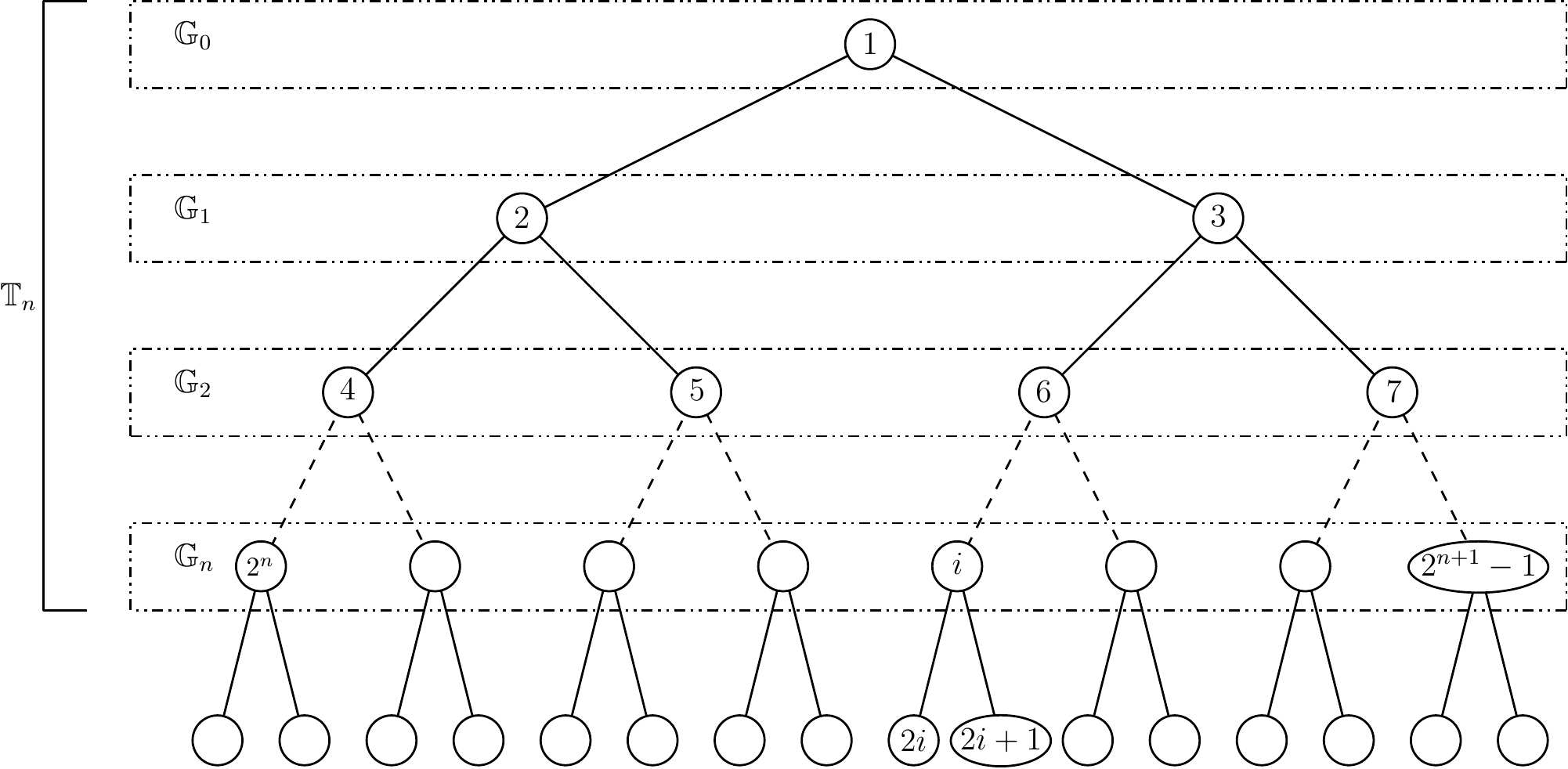}
\caption{The tree associated with the BINAR}
\end{figure}


\section{Weighted least-squares estimation}


Denote by $\F = (\calF_n)_{n\geq0}$ the natural filtration associated with the first-order BINAR process, which means that 
$\calF_n$ is the $\sigma$-algebra generated by all individuals up to the $n$-th generation, in other words $\calF_n=\sigma\{X_k,k\in\T_n\}$. We will assume in all the sequel that, for all $n\geq0$ and for all $k\in \G_n$,
\begin{equation*}
\begin{cases}
\E[\veps_{2k}|\calF_n] = c \hspace{20pt} \text{a.s.}\\
 \E[\veps_{2k+1}|\calF_n] = d \hspace{20pt} \text{a.s.}
\end{cases}
\end{equation*}

\noindent Consequently, we deduce from \eqref{orisyst} that, for all $n\geq0$ and for all $k\in \G_n$,
\begin{equation}
\label{linsyst}
\begin{cases}
X_{2k} &= a X_k + c + V_{2k},\\
X_{2k+1} &= b X_k + d + V_{2k+1},
\end{cases}
\end{equation}

\noindent where, $V_{2k} = X_{2k} - \E[X_{2k}|\calF_{n}]$ and $V_{2k+1} = X_{2k+1} - \E[X_{2k+1}|\calF_{n}]$. Therefore, the two relations given by \eqref{linsyst} can be rewritten in the matrix form
\begin{equation}
 \label{matsyst}
 \chi_n = \theta^t\Phi_n + W_n
\end{equation}
where
$$
\begin{array}{ccccc}
 \chi_n = \begin{pmatrix} X_{2n}\\ X_{2n+1} \end{pmatrix},
  & & \Phi_n = \begin{pmatrix} X_{n}\\ 1 \end{pmatrix},
  & & W_n = \begin{pmatrix} V_{2n}\\ V_{2n+1} \end{pmatrix},
\end{array}
$$
and the matrix parameter
$$\theta = \begin{pmatrix} a & b \\ c & d \end{pmatrix}.$$

\noindent Our goal is to estimate $\theta$ from the observation of all individuals up to $\T_n$. We propose 
to make use of the WLS estimator $\wh{\theta}_n$ of $\theta$ which minimizes
$$\Delta_n(\theta) = \frac12\sum_{k \in \T_{n-1}} \frac1{c_k}\|\chi_k-\theta^t \Phi_k\|^2$$

\noindent where the choice of the weighting sequence $(c_n)_{n\geq1}$ is crucial. We shall choose $c_n=1+X_n$ and we will go back to this suitable choice in Section \ref{martappro}. Consequently, we obviously have for all $n\geq1$
\begin{equation}
\label{eqest}
\wh{\theta}_n = S_{n-1}^{-1} \sum_{k\in\T_{n-1}} \frac1{c_k}\Phi_k\chi_k^t
\end{equation}

\noindent where
$$ S_n = \sum_{k\in\T_n} \frac1{c_k} \Phi_k \Phi_k^t.$$

\noindent In order to avoid useless invertibility assumption, we shall assume, without loss of generality, that for all $n\geq0$, $S_n$ 
is invertible. Otherwise, we only have to add the identity matrix of order 2, $I_2$ to $S_n$. In all what follows, we shall make a slight 
abuse of notation by identifying $\theta$ as well as $\wh\theta_n$ to
$$\begin{array}{ccccc}
 \text{vec}(\theta) = \begin{pmatrix} a \\ c \\  b \\ d \end{pmatrix} & & \text{and} & & 
 \text{vec}(\wh \theta_n) = \begin{pmatrix} \wh{a}_n \\ \wh{c}_n \\ \wh{b}_n \\ \wh{d}_n \end{pmatrix}.
\end{array}$$

\noindent Therefore, we deduce from \eqref{eqest} that
\begin{align}
 \wh{\theta}_n &= \Sigma_{n-1}^{-1} \sum_{k\in\T_n-1} \frac1{c_k}\text{vec}(\Phi_k\chi_k^t) \nonumber, \\
  &= \Sigma_{n-1}^{-1} \sum_{k\in\T_n-1} \frac1{c_k} \begin{pmatrix} X_kX_{2k} \\ X_{2k} \\ X_kX_{2k+1} \\ X_{2k+1} \end{pmatrix} \nonumber
\end{align}

\noindent where $\Sigma_n = I_2 \otimes S_n$ and $\otimes$ stands for the standard Kronecker product. Consequently, \eqref{matsyst} yields to
\begin{align}
 \wh{\theta}_n -\theta &= \Sigma_{n-1}^{-1} \sum_{k\in\T_{n-1}} \frac1{c_k}\text{vec}(\Phi_k W_k^t), \nonumber \\
  \label{diffest} &= \Sigma_{n-1}^{-1} \sum_{k\in\T_{n-1}} \frac1{c_k} \begin{pmatrix} X_kV_{2k} \\ V_{2k} \\ X_kV_{2k+1} 
									    \\ V_{2k+1} \end{pmatrix}.
\end{align}

\noindent In all the sequel, we shall make use of the following moment hypotheses.

\begin{enumerate}[\bf{({H}.}1)]
 \item  \label{eps1} For all $n\geq0$ and for all $k\in \G_n$
$$\begin{array}{ccccccc}
  \E[\veps_{2k}|\calF_n] = c & & \text{and} & & \E[\veps_{2k+1}|\calF_n] = d & & \text{a.s.}
\end{array}$$
 \item \label{eps2} For all $n\geq0$ and for all $k\in \G_n$
$$\begin{array}{ccccccc}
  \Var[\veps_{2k}|\calF_n] = \sigma_c^2 > 0 & \text{and} &  \Var[\veps_{2k+1}|\calF_n] = \sigma_d^2 > 0 & & \text{a.s.}
\end{array}$$
 \item \label{eps11} For all $n\geq0$ and for all $k,l\in\G_{n+1}$, if $[k/2] \neq [l/2]$,
$\veps_k$ and $\veps_l$ are conditionally independent given $\calF_n$, while otherwise it exists $\rho^2 < \sigma_c^2 \sigma_d^2$ such that, for all $k\in\G_n$
$$\E[(\veps_{2k}-c)(\veps_{2k+1}-d)|\calF_n] = \rho \hspace{20pt} \text{ a.s.}$$
 \item \label{eps4} One can find $\mu_c^4>\sigma_c^4$ and $\mu_d^4>\sigma_d^4$ such that, for all $n\geq0$ and for all $k\in \G_n$
$$\begin{array}{ccccccc}
  \E\left[\left(\veps_{2k}-c\right)^4|\calF_n\right] = \mu_c^4 & & \text{and} & & \E\left[\left(\veps_{2k+1}-d\right)^4|\calF_n\right] = \mu_d^4 & & \text{a.s.}
\end{array}$$
In addition, it exists $\nu^4 \leq \mu_c^4 \mu_d^4$ such that, for all $k\in\G_n$
$$\E[(\veps_{2k}-c)^2(\veps_{2k+1}-d)^2|\calF_n] = \nu^2 \hspace{20pt} \text{ a.s.}$$
 \item \label{eps8} One can find $\tau_c^6>0$ and $\tau_d^6>0$ such that
$$\sup_{n\geq1}\sup_{k\in\G_n}\E[\veps_{2k}^6|\calF_n]=\tau_c^6 \hspace{20pt} \text{and} \hspace{20pt} \sup_{n\geq1}\sup_{k\in\G_n}\E[\veps_{2k+1}^6|\calF_n]=\tau_d^6 \hspace{20pt} \text{a.s.}$$
$$\sup_{n\geq2}\E[\veps_n^8]<\infty$$
\end{enumerate}

\vspace{25pt}

\noindent It follows from hypothesis \H{eps1} that $V_{2n}$ and $V_{2n+1}$ can be rewritten as
$$\begin{array}{ccc}
   \displaystyle V_{2n} = \sum_{i=1}^{X_n}(Y_{n,i}-a) + (\veps_{2n}-c) & \text{ and } & \displaystyle V_{2n+1} = \sum_{i=1}^{X_n}(Z_{n,i}-b) + (\veps_{2n}-d).
  \end{array}
$$

\noindent Hence, under assumption \H{eps2}, we have for all $n\geq0$ and for all $k\in\G_n$
\begin{eqnarray}
   \E[V_{2k}^2|\calF_{n}] = \sigma_a^2X_k + \sigma_c^2 & \text{ and } & \E[V_{2k+1}^2|\calF_{n}] = \sigma_b^2X_k + \sigma_d^2 \text{\hspace{20pt} a.s.} \label{espV2}
  \end{eqnarray}

\noindent Consequently, if we choose $c_n=1+X_n$ for all $n\geq1$, we clearly have for all $k\in\G_n$
$$\begin{array}{cccc}
\E\left[\left.V_{2k}^2\right|\calF_{n}\right] \leq \max(\sigma_a^2,\sigma_c^2)c_k & \text{ and } &
\E\left[\left.V_{2k+1}^2\right|\calF_{n}\right] \leq \max(\sigma_b^2,\sigma_d^2)c_k &  \text{ a.s.}
\end{array}$$

\noindent It is exactly the reason why we have chosen this weighting sequence into \eqref{eqest}. Similar WLS estimation approach for branching processes with immigration may be found in \cite{WeiWinnicki} and \cite{Winnicki}. We can also observe that, for all $k\in\G_n$, under the assumption \H{eps11}
$$\begin{array}{cc} \rho = \E[V_{2k}V_{2k+1}|\calF_n] & \text{ a.s.} \end{array}$$
Hence, we propose to estimate the conditional covariance $\rho$ by
\begin{equation}
\wh \rho_n = \frac1{|\T_{n-1}|} \sum_{k\in\T_{n-1}} \wh V_{2k} \wh V_{2k+1}
\end{equation}
where for all $k\in\G_n$, 
\begin{equation*}
\begin{cases}
\wh V_{2k} &= X_{2k} - \wh a_n X_k - \wh c_n,\vspace{1ex}\\
 \wh V_{2k+1} &= X_{2k+1} - \wh b_n X_k - \wh d_n.
\end{cases}
\end{equation*}
For all $n\geq1$, denote $v_{2n} = V_{2n}^2 - \E[V_{2n}^2 | \calF_n]$. We deduce from \eqref{espV2} that for all $n\geq1$ 
$$V_{2n}^2 = \eta^t \Phi_n + v_{2n}$$

\noindent where $\eta^t = \begin{pmatrix} \sigma_a^2 & \sigma_c^2 \end{pmatrix}$. It leads us to estimate the vector of variances $\eta$ by the WLS estimator
\begin{equation} \wh \eta_n = Q_{n-1}^{-1} \sum_{k\in \T_{n-1}} \frac{1}{d_k} \wh V_{2k}^2 \Phi_k \label{esteta} \end{equation}
where 
$$ Q_n = \sum_{k\in\T_n} \frac1{d_k} \Phi_k \Phi_k^t$$

\noindent and the weighting sequence $(d_n)_{n\geq1}$ is given, for all $n\geq1$, by $d_n=(1+X_n)^2$. This choice is due to the fact that for all $n\geq1$ and for all $k\in\G_n$
\begin{align}
\E[v_{2k}^2|\calF_n] &= \E[V_{2k}^4|\calF_n] - \left(\E[V_{2k}^2|\calF_n]\right)^2 \hspace{20pt} \text{ a.s.}\nonumber\\
&= 2 \sigma_a^4X_k^2+(\mu_a^4-3\sigma_a^4 +4\sigma_a^2\sigma_c^2)X_k + \mu_c^4-\sigma_c^4 \hspace{20pt} \text{ a.s.} \label{Ev2k}
\end{align}

\noindent where we recall that $\mu_a^4$ is the fourth-order centered moment of $(Y_{n,i})$. Consequently, as $d_n\geq1$, we clearly have for all $n\geq1$ and for all $k\in\G_n$
$$\E[v_{2k}^2|\calF_n] \leq (\mu_a^4-\sigma_a^4+4\sigma_a^2\sigma_c^2 +\mu_c^4-\sigma_c^4) d_k \hspace{20pt} \text{ a.s.}$$

\noindent We have a similar WLS estimator $\wh\zeta_n$ of the vector of variances $\zeta^t =\begin{pmatrix} \sigma_b^2 & \sigma_d^2 \end{pmatrix}$ by replacing $\wh V_{2k}^2$ by $\wh V_{2k+1}^2$ into \eqref{esteta}.


\section{A martingale approach}\label{martappro}


In order to establish all the asymptotic properties of our estimators, we shall make use of a martingale approach. For all $n\geq1$, denote
$$ M_n = \sum_{k\in\T_{n-1}} \frac1{c_k} \begin{pmatrix} X_kV_{2k} \\ V_{2k} \\ X_kV_{2k+1} \\ V_{2k+1} \end{pmatrix}. $$

\noindent We can clearly rewrite \eqref{diffest} as
\begin{equation}\wh\theta_n - \theta = \Sigma_{n-1}^{-1} M_n.\label{difftheta}\end{equation}

\noindent As in \cite{BercuBDSAGP}, we make use of the notation $M_n$ since it appears that $(M_n)_{n\geq1}$ a martingale. This fact is a crucial point of our study and it justifies the vector notation since most of asymptotic results for martingales were established for vector-valued martingales. Let us rewrite $M_n$ in order to emphasize its martingale quality. Let $\Psi_n = I_2 \otimes \varphi_n$ 
where $\varphi_n$ is the matrix of dimension $2\times2^n$ given by
\begin{equation*}\varphi_n = \begin{pmatrix} \displaystyle \frac{X_{2^n}}{\sqrt{c_{2^n}}} & \displaystyle \frac{X_{2^n+1}}{\sqrt{c_{2^n+1}}} & \displaystyle \hdots & 
		\displaystyle \frac{X_{2^{n+1}-1}}{\sqrt{c_{2^{n+1}-1}}} \vspace{5pt} \\
	    \displaystyle \frac1{\sqrt{c_{2^n}}} & \displaystyle \frac1{\sqrt{c_{2^n+1}}} & \displaystyle \hdots & 
		\displaystyle \frac1{\sqrt{c_{2^{n+1}-1}}}
\end{pmatrix}.\end{equation*}

\noindent It represents the individuals of the $n$-th generation which is also the collection of all $\Phi_k/\sqrt{c_k}$ where $k$ belongs to $\G_n$. Let $\xi_n$ 
be the random vector of dimension $2^n$
$$\xi_n^t = \begin{pmatrix} \displaystyle\frac{V_{2^n}}{\sqrt{c_{2^{n-1}}}} & \displaystyle\frac{V_{2^n+2}}{\sqrt{c_{2^{n-1}+1}}} & \hdots &
			   \displaystyle\frac{V_{2^{n+1}-2}}{\sqrt{c_{2^{n}-1}}} &  \displaystyle\frac{V_{2^n+1}}{\sqrt{c_{2^{n-1}}}} &
			  \displaystyle\frac{V_{2^n+3}}{\sqrt{c_{2^{n-1}+1}}} & \hdots & \displaystyle\frac{V_{2^{n+1}-1}}{\sqrt{c_{2^{n}-1}}}
	    \end{pmatrix}
$$

\noindent The vector $\xi_n$ gathers the noise variables of $\G_n$. The special ordering separating odd and even indices has been made in \cite{BercuBDSAGP}
so that $M_n$ can be written as
$$M_n = \sum_{k=1}^n \Psi_{k-1} \xi_k$$

\noindent Under \H{eps1}, we clearly have for all $n\geq0$, $\E[\xi_{n+1}|\calF_n] = 0$ a.s.~and $\Psi_n$ is $\calF_n$-measurable. In addition it is 
not hard to see that under \H{eps1} to \H{eps11}, $(M_n)$ is a locally square integrable vector martingale with increasing 
process given, for all $n\geq1$, by
\begin{align}
<\!M\!>_n &= \sum_{k=0}^{n-1} \Psi_k \E[\xi_{k+1}\xi_{k+1}^t|\calF_k]\Psi_k^t =\sum_{k=0}^{n-1} L_k \hspace{20pt} \text{a.s.} \label{defcrochetM}
\end{align}

\noindent where 
\begin{equation}\label{defLk}
    L_k = \sum_{i\in\G_{k}} \frac1{c_i^2} \begin{pmatrix} \sigma_a^2X_i+\sigma_c^2 & \rho \\ \rho & \sigma_b^2X_i+\sigma_d^2 \end{pmatrix}\otimes\begin{pmatrix} X_i^2 & X_i \\ X_i & 1\end{pmatrix}.
\end{equation}

\noindent It is necessary to establish the convergence of $<\!M\!>_n$, properly normalized, in order to prove the asymptotic results for our BINAR estimators $\wh\theta_n$, $\wh\eta_n$ and $\wh\zeta_n$. Since the sizes of $\Psi_n$ and $\xi_n$ double at each generation, we have to adapt the proof of 
vector-valued martingale convergence given in \cite{Duflo} to our framework.


\section{Main results} \label{mainresults}


In all the sequel, we will assume that $\Pro_{\veps_{2n}}$ and $\Pro_{\veps_{2n+1}}$ do not depend on $n$. However, we shall get rid of the standard assumption commonly used in the INAR literature that the offspring sequences $(Y_{n,i})$ and $(Z_{n,i})$ share the same Bernoulli distribution. The only assumption that we will use here is that the offspring sequences $(Y_{n,i})$ and $(Z_{n,i})$ admit eighth-order moments. We have to introduce some more notations in order to state our main results. From the original process $(X_n)_{n\geq1}$, we shall define a new process $(Y_n)_{n\geq1}$ recursively defined by $Y_1=X_1$, and if $Y_n=X_k$ with $n,k\geq1$, then 
$$Y_{n+1} = X_{2k+\kappa_n}$$

\noindent where $(\kappa_n)_{n\geq1}$ is a sequence of i.i.d.~random variables with Bernoulli $\calB\left(1/2\right)$ distribution. Such a construction may be found in \cite{Guyon} for the asymptotic analysis of BAR processes. The process $(Y_n)$ gathers the values of the original process $(X_n)$ along the random branch of the binary tree $(\T_n)$ given by $(\kappa_n)$. Denote by $k_n$ the unique $k\geq1$ such that $Y_n=X_k$. Then, for all $n\geq1$, we have
\begin{equation} Y_{n+1} = a_{n+1} \circ Y_{n} + e_{n+1} \label{defYn1}\end{equation}
where
\begin{equation}
\begin{array}{ccc}
 a_{n+1} = \begin{cases}
a \text{ if } \kappa_n = 0 \\
   b \text{ otherwise}
\end{cases}
 \hspace{25pt} \text{and} \hspace{25pt} &
 e_n = \veps_{k_n}.
\end{array} \label{defYn2}
\end{equation}

\begin{Lem} \label{CVYn}
 Assume that $(\veps_n)$ satisfies \H{eps1} to \H{eps4}. Then, we have
\begin{equation*}
 Y_n \liml T 
\end{equation*}
where $T$ is a positive non degenerate integer-valued random variable with $\E[T^3]<\infty$.
\end{Lem}

\noindent Denote
$\calC_{3}^1(\R_+) = \Bigl\{f\in\calC^1(\R_+,\R)\big|\exists \gamma >0, \forall x\geq0, (|f'(x)| + |f(x)|) \leq \gamma(1+x^3)\Bigl\}$.

\begin{Lem}\label{LFGN}
Assume that $(\veps_n)$ satisfies \H{eps1} to 
\H{eps8}. Then, for all $f\in \calC_{3}^1(\R_+)$, we have
$$\lim_{n\to\infty} \frac1{|\T_n|} \sum_{k\in \T_n} f(X_k) = \E[f(T)] \hspace{20pt}  \text{a.s.}$$
\end{Lem}

\begin{Prop} \label{cvcrochet}
 Assume that $(\veps_n)$ satisfies \H{eps1} to \H{eps8}. Then, we have 
\begin{equation} \label{limcrochet}
 \lim_{n\to\infty} \frac{<\!M\!>_n}{|\T_{n-1}|} = L \hspace{20pt} \text{ a.s.}
\end{equation}
\noindent where $L$ is the positive definite matrix given by
\begin{equation*} L = 
\E\left[\frac1{(1+T)^2} \begin{pmatrix} \sigma_a^2T+\sigma_c^2 & \rho \\ \rho & \sigma_b^2T+\sigma_d^2 \end{pmatrix} \otimes \begin{pmatrix} T^2 & T \\ T & 1 \end{pmatrix} \right].
\end{equation*}

\end{Prop}

\noindent Our first result deals with the almost sure convergence of our WLS estimator $\wh\theta_n$.

\begin{Theo} \label{CVpstheta}
 Assume that $(\veps_n)$ satisfies \H{eps1} to \H{eps8}. Then, $\wh\theta_n$ converges almost surely 
to $\theta$ with the rate of convergence
\begin{equation}
\|\wh\theta_n-\theta\|^2 =  O\left(\frac{n}{|\T_{n-1}|}\right) \hspace{20pt} \text{ a.s.} \label{rate}
\end{equation}
\noindent In addition, we also have the quadratic strong law
\begin{equation}\label{quadratic1}
 \lim_{n\to\infty} \frac1n \sum_{k=1}^n |\T_{k-1}|(\wh \theta_k -\theta)^t \Lambda (\wh \theta_k -\theta) = tr(\Lambda^{-1/2}L\Lambda^{-1/2}) \hspace{20pt} \text{ a.s.}\\
\end{equation}
\noindent where
\begin{equation}\label{defA}
\Lambda = I_2\otimes A \hspace{20pt} \text{ and } \hspace{20pt} A = \E\left[\frac1{1+T}\begin{pmatrix} T^2 & T \\ T & 1 \end{pmatrix}\right].
\end{equation}

\end{Theo}

\noindent Our second result concerns the almost sure asymptotic properties of our WLS variance and covariance estimators $\wh \eta_n$, $\wh \zeta_n$ and $\wh \rho_n$. Let
\begin{align*}
\eta_n  &= Q_{n-1}^{-1} \sum_{k\in \T_{n-1}} \frac{1}{d_k} V_{2k}^2 \Phi_k,\\
\zeta_n &= Q_{n-1}^{-1} \sum_{k\in \T_{n-1}} \frac{1}{d_k} V_{2k+1}^2 \Phi_k,\\
\rho_n &= \frac1{|\T_{n-1}|}\sum_{k\in\T_{n-1}} V_{2k}V_{2k+1}.
\end{align*}

\begin{Theo}\label{CVpsvar}
 Assume that $(\veps_n)$ satisfies \H{eps1} to \H{eps8}. Then, $\wh \eta_n$ and $ \wh \zeta_n$ converge almost surely to $\eta$ and 
$\zeta$ respectively. More precisely,
\begin{align}
\|\wh \eta_{n} - \eta_{n}\|  &= \calO\left(\frac{n}{|\T_{n-1}|}\right) \hspace{20pt} \text{ a.s.}\label{vitesseeta} \\
\|\wh \zeta_{n} - \zeta_{n}\| &= \calO\left(\frac{n}{|\T_{n-1}|}\right) \hspace{20pt} \text{ a.s.}\label{vitesseetad} 
\end{align}

\noindent In addition, $\wh \rho_n$ converges almost surely to $\rho$ with
\begin{equation}
\wh \rho_n - \rho_n = \calO\left(\frac n{|\T_{n-1}|}\right) \hspace{20pt} \text{ a.s.} \label{vitesserho}
\end{equation}

\end{Theo}

\begin{Rem}\label{remrate}
We also have the rates of convergence
$$\|\wh \eta_{n} - \eta\|^2 =\calO\left(\frac{n}{|\T_{n-1}|}\right),~~ \|\wh \zeta_{n} - \zeta\|^2 =\calO\left(\frac{n}{|\T_{n-1}|}\right),~~ (\wh \rho_{n} - \rho)^2 =\calO\left(\frac{n}{|\T_{n-1}|}\right) ~~~ a.s.$$
\end{Rem}

\noindent Our last result is devoted to the asymptotic normality of our WLS estimators $\wh\theta_n$, $\wh \eta_n$, $\wh \zeta_n$ and $\wh \rho_n$.

\begin{Theo}\label{TCL}
 Assume that $(\veps_n)$ satisfies \H{eps1} to \H{eps8}. Then, we have the asymptotic normality
\begin{equation} \label{TCLtheta}
 \sqrt{|\T_{n-1}|}(\wh\theta_n - \theta) \liml \calN(0,(I_2\otimes A^{-1})L(I_2\otimes A^{-1})).
\end{equation}

\noindent In addition, we also have
\begin{eqnarray}
\sqrt{|\T_{n-1}|}\left(\wh \eta_{n} - \eta\right)  \liml \calN(0,B^{-1} M_{ac} B^{-1}) \label{TCLeta},\\
\sqrt{|\T_{n-1}|}\left(\wh \zeta_{n} - \zeta\right) \liml \calN(0,B^{-1} M_{bd} B^{-1}) \label{TCLetad},
\end{eqnarray}
\noindent where
$$B = \E\left[\frac1{(1+T)^2} \begin{pmatrix} T^2 & T \\ T & 1 \end{pmatrix} \right],$$
$$M_{ab} = \E\left[\frac{2\sigma_a^4T^2+(\mu_a^4-3\sigma_a^4+4\sigma_a^2\sigma_c^2)T+\mu_c^4-\sigma_c^4}{(1+T)^4} \begin{pmatrix} T^2 & T \\ T & 1 \end{pmatrix} \right],$$
$$M_{bd} = \E\left[\frac{2\sigma_b^4T^2+(\mu_b^4-3\sigma_b^4+4\sigma_b^2\sigma_d^2)T+\mu_d^4-\sigma_d^4}{(1+T)^4} \begin{pmatrix} T^2 & T \\ T & 1 \end{pmatrix} \right].$$
Finally,
\begin{equation} 
\sqrt{|\T_{n-1}|}  \left(\wh \rho_n - \rho\right) \liml \calN\left(0,\sigma_\rho^2 \right) \label{TCLrho}
\end{equation}

\noindent where 
\begin{equation}\label{sigmarho}
\sigma_\rho^2 = \sigma_a^2 \sigma_b^2 \E[T^2] + \left(\sigma_a^2 \sigma_d^2 + \sigma_b^2 \sigma_c^2\right)\frac{\overline c}{1-\overline a} + \nu^2 - \rho^2,
\end{equation}
$$\E[T^2] = \frac{\Upsilon\overline{c}}{1-\overline{a}} + \frac{\overline{c^2} - \Upsilon \overline{c}}{1-\overline{a^2}} + \frac{2\overline{a}(\overline{c}^2)}{(1-\overline{a})(1-\overline{a}^2)},$$

$$\Upsilon= \frac{\sigma_a^2 + \sigma_b^2}{2(\overline a  - \overline{a^2})}, \hspace{20pt} \overline {a} = \frac{a+ b}{2}, \hspace{20pt} \overline {a^2} = \frac{a^2 + b^2}{2},$$ 
$$\overline c = \frac{c+d}{2}, \hspace{20pt}  \overline{c^2} = \frac{\sigma_c^2 + \sigma_d^2 + c^2 + d^2}2.$$

\end{Theo}

\noindent The rest of the paper is dedicated to the proof of our main results.


\section{Proof of Lemma \ref{CVYn}} \label{preuve lemme CVYn}


We can reformulate \eqref{defYn1} and \eqref{defYn2} as
$$Y_n = a_n \circ a_{n-1} \circ \hdots \circ a_2 \circ Y_1 + \sum_{k=2}^{n-1} a_n \circ a_{n-1} \circ \hdots \circ a_{k+1} \circ e_k 
      + e_n.$$

\noindent We already made the assumption that $\Pro_{\veps_{2n}}$ and $\Pro_{\veps_{2n+1}}$ do not depend on $n$. Consequently, the couples $(a_k,e_k)$ and $(a_{n-k+2},e_{n-k+2})$ share the same distribution. Hence, for all $n\geq2$, $Y_n$ has the same distribution than the random variable
\begin{align*}
 Z_n &= a_2 \circ \hdots \circ a_n \circ Y_1 + \sum_{k=2}^{n-1} a_2 \circ a_3 \circ \hdots \circ a_{n-k+1} \circ e_{n-k+2} 
    + e_2,\\
	      &= a_2 \circ \hdots \circ a_n \circ Y_1 + \sum_{k=3}^{n} a_2 \circ a_3 \circ \hdots \circ a_{k-1} \circ e_{k} 
    + e_2.
\end{align*}

\noindent For the sake of simplicity, we will denote 
\begin{equation}Z_n = a_2 \circ \hdots \circ a_n \circ Y_1 + \sum_{k=2}^{n} a_2 \circ a_3 \circ \hdots \circ a_{k-1} \circ e_{k}.\label{defZn}\end{equation}

\noindent For all $n\geq 2$ and for all $2\leq k \leq n$, let 
$$\Sigma_n^{n-k+2}  = a_k \circ \hdots \circ a_n \circ Y_1$$

\noindent with $\Sigma_n^n = a_2 \circ \hdots \circ a_n \circ Y_1$  and $\Sigma_n^1=Y_1$. We clearly have $\Sigma_n^{n-k+2}= a_k \circ \Sigma_n^{n-k+1}$. Consequently, it follows from the tower property of the conditional expectation that
\begin{align*}
\E[\Sigma_n^n] &= \E[a_2 \circ \Sigma_n^{n-1}] = \left(\E[a \circ \Sigma_n^{n-1}]\Pro(a_2=a) +\E[b \circ \Sigma_n^{n-1}]\Pro(a_2=b)\right),\\
&= \frac12\left(\E\left[\E\left[\left.\sum_{i=1}^{\Sigma_n^{n-1}} Y_{2,i} \right| \Sigma_n^{n-1}\right]\right] + \E\left[\E\left[\left.\sum_{i=1}^{\Sigma_n^{n-1}} Z_{2,i} \right| \Sigma_n^{n-1}\right]\right]\right),\\
&= \frac12\left(\E\left[\sum_{i=1}^{\Sigma_n^{n-1}}\E\left[ Y_{2,i}\right]\right] + \E\left[\sum_{i=1}^{\Sigma_n^{n-1}}\E\left[ Z_{2,i} \right]\right]\right),\\
&= \frac12\left(\E[a\Sigma_n^{n-1}]+\E[b\Sigma_n^{n-1}]\right) = \overline a \E[\Sigma_n^{n-1}] = \hdots\\
&=\overline a ^{n-1} \E[\Sigma_n^{1}] = \overline a ^{n-1} \E[Y_1].
\end{align*}

\noindent The stability hypothesis $0<\max(a,b)<1$ implies that $0<\overline a<1$ which leads to
$$\sum_{n=2}^\infty \E[\Sigma_n^n]=\E[Y_1] \sum_{n=2}^\infty \overline a ^{n-1} = \frac{\E[Y_1]\overline a}{1-\overline a}.$$

\noindent Then, we obtain from the monotone convergence theorem that 
\begin{equation}\label{limSigman}\lim_{n\to\infty} \Sigma_n^n =0 \hspace{20pt} \text{ a.s.}\end{equation}

\noindent It now remains to study the right-hand side sum in \eqref{defZn}. For all $n\geq2$, denote
$$T_n = \sum_{k=2}^n a_2\circ \hdots \circ a_{k-1} \circ e_k.$$

\noindent By the same calculation as before, we have for all $n\geq2$
$$\E[T_n] = \sum_{k=2}^n \overline a ^{k-2} \E[e_k] = \overline c \sum_{k=0}^{n-2} \overline a ^{k},$$
which implies that
$$\lim_{n\to\infty} \E[T_n] = \frac{\overline c}{1-\overline a}.$$

\noindent Hence, we deduce once again from the monotone convergence theorem that the positive increasing sequence $(T_n)_{n\geq2}$ converges almost surely to
$$T = \sum_{k=2}^\infty a_2\circ \hdots \circ a_{k-1} \circ e_k$$

\noindent which is almost surely finite as $\E[T]<\infty$. Therefore, we can conclude from \eqref{defZn} and \eqref{limSigman} that
$$\lim_{n\to\infty} Z_n =T \hspace{20pt} \text{ a.s.}$$

\noindent leading to 
$$Y_n \liml T.$$

\noindent It only remains to prove that $T$ is not degenerate. Let us write $T_n$ as
$$T_n = e_2 + \sum_{k=3}^n a_2\circ \hdots \circ a_{k-1} \circ e_k.$$
\noindent Since $e_2$ is independent of $(a_2\circ \hdots \circ a_{k-1} \circ e_k)_{k\geq3}$, we have
\begin{equation}\label{VarTn}\Var(T_n) = \Var(e_2) + \Var\left(\sum_{k=3}^n a_2\circ \hdots \circ a_{k-1} \circ e_k\right) \geq \Var(e_2).\end{equation}

\noindent Moreover, it is easy to see that
\begin{equation}\label{Varpos}\Var(e_2) = \overline{c^2} - \overline{c}^2 = \frac{\sigma_c^2+\sigma_d^2}2 + \frac{(c-d)^2}4 \geq \frac{\sigma_c^2+\sigma_d^2}2 >0.\end{equation}

\noindent Consequently, as $(T_n)$ is a positive non-decreasing sequence, we obtain from \eqref{VarTn} and \eqref{Varpos} that 
$$\Var(T)=\lim_{n\to\infty} \Var(T_n) \geq \Var(e_2)>0.$$

\noindent Finally, let us prove that $\E[T^3] < \infty$. First of all, we already saw that
$$\E[a_2 \circ \hdots \circ a_n \circ e_{n+1}] = \overline a ^{n-1} \E[e_2] = \overline a ^{n-1} \overline c.$$
In addition,
\begin{align*}
\E[(\Sigma_n^n)^2] &= \frac12 \left(\E\left[(a \circ \Sigma_n^{n-1})^2\right] + \E\left[(b \circ \Sigma_n^{n-1})^2\right]\right),\\
&= \frac12\left(\E\left[\E\left[\left.\left(\sum_{i=1}^{\Sigma_n^{n-1}} Y_{2,i}\right)^2 \right| \Sigma_n^{n-1}\right]\right] + \E\left[\E\left[\left.\left(\sum_{i=1}^{\Sigma_n^{n-1}} Z_{2,i}\right)^2 \right| \Sigma_n^{n-1}\right]\right]\right),
\end{align*}
and the first expectation is
\begin{align*}
\E\left[\E\left[\left.\left(\sum_{i=1}^{\Sigma_n^{n-1}} Y_{2,i}\right)^2 \right| \Sigma_n^{n-1}\right]\right] &= \E\left[\E\left[\left.\sum_{i=1}^{\Sigma_n^{n-1}} Y_{2,i}^2 + \sum_{i=1}^{\Sigma_n^{n-1}} \sum_{\substack{j=1 \\ j\neq i}}^{\Sigma_n^{n-1}} Y_{2,i} Y_{2,j} \right| \Sigma_n^{n-1}\right]\right],\\
	&= \E\left[\sum_{i=1}^{\Sigma_n^{n-1}}\E[ Y_{2,i}^2] + \sum_{i=1}^{\Sigma_n^{n-1}} \sum_{\substack{j=1 \\ j\neq i}}^{\Sigma_n^{n-1}} \E[Y_{2,i}]\E[ Y_{2,j}]\right],\\
	&= \E[\Sigma_n^{n-1}(\sigma_a^2 +a^2) + \Sigma_n^{n-1}(\Sigma_n^{n-1}-1)a^2],\\
	&=  \E[\Sigma_n^{n-1}]\sigma_a^2 + a^2 \E[(\Sigma_n^{n-1})^2].
\end{align*}
Since the computation of the second expectation is exactly the same, we obtain
\begin{align*}
\E[(\Sigma_n^n)^2] &= \E[\Sigma_n^{n-1}] \frac{\sigma_a^2 + \sigma_b^2}2 + \overline{a^2} \E[(\Sigma_n^{n-1})^2],\\
	&= \overline a ^{n-2} \frac{\sigma_a^2 + \sigma_b^2}2 \E[Y_1] + \overline{a^2} \E[(\Sigma_n^{n-1})^2] = \hdots\\
	&= \left(\sum_{i=0}^{n-2}\overline a ^{n-i-2} \overline{a^2}^i\right) \frac{\sigma_a^2 + \sigma_b^2}2 \E[Y_1] + \overline{a^2}^{n-1} \E[(\Sigma_n^{1})^2],\\
	&=  \frac{\overline a ^{n-1} - \overline{a^2}^{n-1}}{\overline a  - \overline{a^2}}\frac{\sigma_a^2 + \sigma_b^2}2 \E[Y_1] + \overline{a^2}^{n-1} \E[Y_1^2],\\
	&=(\overline a ^{n-1} - \overline{a^2}^{n-1}) \Upsilon \E[Y_1] + \overline{a^2}^{n-1}\E[Y_1^2]
\end{align*}
where 
$$\Upsilon = \displaystyle\frac{\sigma_a^2 + \sigma_b^2}{2(\overline a  - \overline{a^2})}.$$
In the same way, we can prove that 
$$\E_x[(a_2 \circ \hdots \circ a_n \circ e_{n+1})^2] = (\overline a ^{n-1} - \overline{a^2}^{n-1}) \Upsilon \overline c + \overline{a^2}^{n-1}\overline{c^2}.$$
Consequently, as $(e_n)$ is an integer-valued random variable,
$$\E_x[(a_2 \circ \hdots \circ a_n \circ e_{n+1})^2] \leq \overline a ^{n-1} (\Upsilon \overline c + \overline{c^2})\leq \overline a ^{n-1} (\Upsilon +1) \overline{c^2}.$$
Furthermore, we obtain from tedious but straightforward calculations that it exists some constant $\xi>0$ such that for all $2\leq p\leq8$
\begin{equation} \label{majEl}
\E_x[(a_2 \circ \hdots \circ a_n \circ e_{n+1})^p]\leq \xi \E[e_2^p] \overline a ^{n-1}.
\end{equation}
One can observe that the constant $\xi$ only depends on the moments of $(Y_{n,i})$ and $(Z_{n,i})$ up to order 8. Hence, as $0<\overline a <1$, we deduce from \eqref{majEl} and the triangle inequality that
\begin{align*}
\E[T^3]^{1/3} &\leq \sum_{k=2}^\infty \E\left[(a_2 \circ \hdots \circ a_{k-1} \circ e_k)^3\right]^{1/3},\\
&\leq \xi^{1/3} \E[e_2^3]^{1/3} \sum_{k=2}^\infty \overline{a}^{(k-2)/3}  <\infty
\end{align*}

\noindent which immediately leads to $\E[T^3]<\infty$.


\section{Proof of Lemma \ref{LFGN}} \label{demoLFGN}


We shall now prove that for all $f\in\calC_{3}^1(\R_+)$,
\begin{equation} \label{limLFGN}
\lim_{n\to\infty} \frac1{|\T_n|} \sum_{k\in\T_n} f(X_k) = \E[f(T)].
\end{equation}
Denote $g=f-\E[f(T)]$,
$$\begin{array}{ccccc}
\displaystyle{\overline M_{\T_n} (f) = \frac1{|\T_n|} \sum_{k\in \T_n} f(X_k)} & & \text{and} & & \displaystyle{\overline M _{\G_n} (f) = \frac1{|\G_n|} \sum_{k\in \G_n} f(X_k)}.
\end{array}$$
Via Lemma A.2 of \cite{BercuBDSAGP}, it is only necessary to prove that
$$\lim_{n\to\infty} \frac1{|\G_n|} \sum_{k\in\G_n} g(X_k) = 0 \hspace{20pt} \text{a.s.}$$
We shall follow the induced Markov chain approach, originally proposed by Guyon in \cite{Guyon}. Let $Q$ be the transition probability of $(Y_n)$, $Q^p$ the $p$-th iterated of $Q$. In addition, denote by $\nu$ the distribution of $Y_1=X_1$ and $\nu Q^p$ the law of $Y_p$. Finally, let $P$ be the transition probability of $(X_n)$ as defined in \cite{Guyon}. We obtain from relation (7) of \cite{Guyon} that for all $n\geq0$
$$\E[\overline M _{\G_n} (g)^2] = \frac1{2^n} \nu Q^n g^2 + \sum_{k=0}^{n-1} \frac1{2^{k+1}} \nu Q^k P(Q^{n-k-1}g \star Q^{n-k-1}g)$$
where, for all $x,y\in\N$, $(f\star g)(x,y) = f(x)g(y)$. Consequently,
\begin{align*}
\sum_{n=0}^\infty \E[\overline M _{\G_n} (g)^2] &= \sum_{n=0}^\infty  \frac1{2^n} \nu Q^n g^2 + \sum_{n=1}^\infty \sum_{k=0}^{n-1} \frac1{2^{k+1}} \nu Q^k P(Q^{n-k-1}g \star Q^{n-k-1}g),\\
	&\leq \sum_{k=0}^\infty \frac1{2^k} \nu Q^k\left(g^2+P\left(\sum_{l=0}^\infty |Q^l g \star Q^l g|\right)\right).
\end{align*}
However, for all $x\in \N$,
$$Q^n g(x) = Q^n f(x) - \E[f(T)] = \E[f(Y_n) - f(T)] = \E_x[f(Z_n)-f(T)]$$
where $Z_n$ is given by \eqref{defZn}. Hence, we deduce from the mean value theorem and Cauchy-Schwarz inequality that
\begin{equation}
|Q^n g(x)| \leq \E_x[W_n|Z_n-T|] \leq \E_x[W_n^2]^{1/2}\E_x[(Z_n-T)^2]^{1/2} \label{majQn}
\end{equation}
where
$$W_n = \sup_{z\in[Z_n,T]}|f'(z)|.$$
By the very definition of $\calC_{3}^1(\R_+)$, one can find some constant $\alpha>0$ such that $|f'(z)|\leq \alpha (1+z^6)$. Hence, it exists some constant $\beta>0$ such that
\begin{align}
\E_x[W_n^2] &\leq \alpha \E_x[1+Z_n^6+T^6] = \alpha (1+ \E_x[Z_n^6] + \E[T^6]) \nonumber,\\
	& \leq \beta(1+x^6). \label{majEWn}
\end{align}

%
%
%

\noindent As a matter of fact, under hypotheses \H{eps1} to \H{eps8}, $\E[T^6] <\infty$ and it exists some constant $\gamma >0$ such that $\E_x[Z_n^6]<\gamma (1+x^6)$. Let us first deal with $\E[T^6]$. The triangle inequality, together with $0<\overline a <1$ and \eqref{majEl} allow us to say that
\begin{equation*}
\E[T^6]^{1/6} \leq \sum_{k=2}^\infty \E\left[(a_2\circ \hdots \circ a_{k-1}\circ e_k)^6\right]^{1/6} \leq \xi^{1/6} \E[e_2^6]^{1/6} \sum_{k=2}^\infty \overline{a}^{(k-2)/6} < \infty
\end{equation*}
which immediately leads to $\E[T^6]<\infty$. One the other hand, we infer from \eqref{defZn} that
\begin{align*} 
\E_x[Z_n^6]^{1/6} &\leq \E_x[(a_2\circ\hdots\circ a_n\circ Y_1)^6]^{1/6} + \sum_{k=2}^{n} \E_x\left[(a_2 \circ a_3 \circ \hdots \circ a_{k-1} \circ e_{k})^6\right]^{1/6}, \\
& \leq \xi^{1/6} \E_x[Y_1^6]^{1/6} \overline a ^{n-1} + \sum_{k=2}^{\infty} \E\left[(a_2 \circ a_3 \circ \hdots \circ a_{k-1} \circ e_{k})^6\right]^{1/6},\\
&\leq \xi^{1/6}  x + \sum_{k=2}^{\infty} \E\left[(a_2 \circ a_3 \circ \hdots \circ a_{k-1} \circ e_{k})^6\right]^{1/6}
\end{align*}
and we have already proved that the sum in the right-hand term is finite. So we can conclude that there exists some constant $\gamma>0$ such that $\E_x[Z_n^6]<\gamma (1+x^6)$. Furthermore 
$$\displaystyle Z_n- T= a_2 \circ \hdots a_n \circ Y_1 - \sum_{k=n}^\infty a_2 \circ \hdots \circ a_k \circ e_{k+1}$$
and the triangle inequality allows us to say that 
$$\E_x[(Z_n-T)^2]^{1/2} \leq \E_x[(a_2 \circ \hdots a_n \circ Y_1)^2]^{1/2} + \sum_{k=n}^\infty \E_x[(a_2 \circ \hdots \circ a_k \circ e_{k+1})^2]^{1/2}.$$

\noindent We already saw in section \ref{preuve lemme CVYn} that
\begin{align*}
\E_x[(a_2 \circ \hdots a_n \circ Y_1)^2] &= (\overline a ^{n-1} - \overline{a^2}^{n-1}) \Upsilon \E_x[Y_1] + \overline{a^2}^{n-1}\E_x[Y_1^2],\\
	&= (\overline a ^{n-1} - \overline{a^2}^{n-1}) \Upsilon x + \overline{a^2}^{n-1}x^2 = x(\Upsilon \overline a ^{n-1} + \overline{a^2}^{n-1}(x-\Upsilon))
\end{align*}
and
$$\E_x[(a_2 \circ \hdots \circ a_k \circ e_{k+1})^2] = (\overline a ^{k-1} - \overline{a^2}^{k-1}) \Upsilon \overline c + \overline{a^2}^{k-1}\overline{c^2}.$$

\noindent Hence
\begin{align*}
 \sum_{k=n}^\infty \E_x[(a_2 \circ \hdots \circ a_k \circ e_{k+1})^2]^{1/2}	&= \sum_{k=n}^\infty \left(\overline a ^{k-1} \Upsilon \overline c + \overline{a^2}^{k-1} \left(\overline{c^2} - \Upsilon \overline c\right)\right)^{1/2},\\
	&\leq \sum_{k=n}^\infty \left(\overline a ^{k-1} \overline c + \overline{a}^{k-1} \left|\overline{c^2} - \Upsilon \overline c\right|\right)^{1/2},\\
	&\leq \sum_{k=n}^\infty \sqrt{\overline a}^{k-1} \delta= \delta \frac{\sqrt{\overline a}^{n-1}}{1-\sqrt{\overline a}}.
\end{align*} 
where 
$$\delta =\sqrt{\max(\overline{c^2},(1+\Upsilon)\overline c - \overline{c^2})}.$$
To sum up, we find that
\begin{align}
 \E_x[(Z_n-T)^2]^{1/2} &\leq \sqrt{x} \left(\Upsilon \overline a ^{n-1}  + \overline{a^2}^{n-1} (x-\Upsilon)\right)^{1/2} +  \frac{\delta}{1-\sqrt{\overline a}}\sqrt{\overline a}^{n-1}, \nonumber\\
	& \leq \begin{cases}
	\sqrt{x} \left(\Upsilon \overline a ^{n-1}  + \overline{a}^{n-1} (x-\Upsilon)\right)^{1/2} +  \displaystyle \frac{\delta}{1-\sqrt{\overline a}}\sqrt{\overline a}^{n-1} & \text{ if } x > \Upsilon,\\
	\displaystyle \sqrt{x} \sqrt{\Upsilon} \sqrt{\overline a}^{n-1}  + \frac{\delta}{1-\sqrt{\overline a}}\sqrt{\overline a}^{n-1} & \text{ if } x \leq \Upsilon,
	\end{cases} \nonumber\\
	&\leq \begin{cases}
	x \sqrt{\overline a} ^{n-1} +  \displaystyle \frac{\delta}{1-\sqrt{\overline a}}\sqrt{\overline a}^{n-1}  & \text{ if } x > \Upsilon,\\
	\displaystyle \frac{1+x}2 \sqrt{\Upsilon}\sqrt{\overline a}^{n-1} + \frac{\delta}{1-\sqrt{\overline a}}\sqrt{\overline a}^{n-1}& \text{ if } x\leq\Upsilon,
	\end{cases} \nonumber\\
	  &\leq \sqrt{\overline a} ^{n-1} (1+x) \left(\frac{\sqrt{\Upsilon}}2 + \frac{\delta}{1-\sqrt{\overline a}} \right). \label{majEZnT}
\end{align}

\noindent Finally, we obtain from \eqref{majQn} together with \eqref{majEWn} and \eqref{majEZnT} that for some constant $\kappa>0$
$$ |Q^n g(x)| \leq \sqrt{\beta} (1 + x^6)^{1/2}\sqrt{\overline a} ^{n-1} (1+x)\left(\frac{\sqrt{\Upsilon}}2 + \frac{\delta}{1-\sqrt{\overline a}} \right)\leq \sqrt{\overline a} ^n \kappa (1 + x^4).$$
Therefore,
$$\Pro\left(\sum_{n=0}^\infty|Q^n g\star Q^n g|\right) \leq \frac{\kappa^2}{1-\overline a} P(h\star h)$$
where, for all $x\in\N$, $h(x)=1+x^4$. We are now in position to prove that
\begin{equation}\label{but7}
\E\left[\sum_{n=0}^\infty \overline M _{\G_n} (g)^2\right] < \infty.
\end{equation}
It is not hard to see that from hypothesis $\H{eps8}$, it exists some constant $\lambda>0$ such that for all $x\in\N$, $P(h\star h) (x) \leq \lambda(1+x^8)$. 
%
Consequently, it exists some constant $\mu>0$ such that
\begin{align}
 \sum_{n=0}^\infty \E\left[\overline M _{\G_n} (g)^2\right] &\leq \sum_{k=0}^\infty \frac1{2^k} \nu Q^k\left(g^2+P\left(\sum_{l=0}^\infty |Q^l g \star Q^l g|\right)\right),\nonumber\\
	  &\leq \sum_{k=0}^\infty \frac1{2^k}\left(\E[g^2(Y_k)] + \frac{\lambda\kappa^2}{1-\overline a}(1+\E[Y_k^8]) \right),\nonumber\\
	&\leq \left(2\mu + \frac{\lambda\kappa^2}{1-\overline a}\right)\left(2+\sum_{k=0}^\infty \E[Y_k^8] \right). \label{majEMG}
\end{align}
Furthermore, we can deduce from \eqref{majEl} that it exists some constant $\zeta$ such that
\begin{align}
\E[Y_n^8]^{1/8} &\leq \E\left[(a_2 \circ \hdots a_n \circ Y_1)^8\right]^{1/8} + \sum_{k=2}^n \E\left[(a_2 \circ \hdots a_{k-1} \circ e_k)^8\right]^{1/8},\nonumber\\
	&\leq \E\left[(a_2 \circ \hdots a_n \circ Y_1)^8\right]^{1/8} + \xi^{1/8} \E[e_2^8]^{1/8} \sum_{k=2}^n \overline a ^{k-2},\nonumber\\
	&\leq \zeta^{1/8} \E[Y_1^8]^{1/8}\overline a ^{n-1} + \frac{\xi^{1/8}\E[e_2^8]^{1/8}}{1-\overline a},\nonumber\\
	&\leq \frac{\zeta^{1/8} \E[Y_1^8]^{1/8} + \xi^{1/8} \E[e_2^8]^{1/8}}{1-\overline a}. \label{majEY8}
\end{align}
Then, \eqref{majEMG} and \eqref{majEY8} immediately lead to \eqref{but7}. Finally, the monotone convergence theorem implies that
$$\lim_{n\to\infty} \overline M _{\G_n} (g) = 0 \hspace{20pt} a.s.$$
which completes the proof of Lemma \ref{LFGN}.

%
%


\section{Proof of Proposition \ref{cvcrochet}}


The almost sure convergence \eqref{limcrochet} immediately follows from \eqref{defcrochetM} and \eqref{defLk} together with Lemma \ref{LFGN}. It only remains to prove that $\det(L)>0$ where the limiting matrix $L$ can be rewritten as
$$L = \E\left[\Gamma \otimes \calB\right]$$
where
$$\begin{array}{ccccc}
 \Gamma = \begin{pmatrix}
           \sigma_a^2 T + \sigma_c^2 & \rho \\ \rho & \sigma_b^2 T + \sigma_d^2
          \end{pmatrix}
 & & \text{ and } & &
 \calB = \begin{pmatrix}
	\displaystyle\frac{T^2}{(1+T)^2} & \displaystyle\frac{T}{(1+T)^2} \vspace{5pt}\\ \displaystyle\frac{T}{(1+T)^2} & \displaystyle\frac{1}{(1+T)^2}
     \end{pmatrix}.
\end{array}$$

\noindent We have
\begin{align}
 L &= \E\left[\begin{pmatrix} \sigma_a^2 T & 0 \\ 0 & \sigma_b^2 T \end{pmatrix} \otimes \calB\right]
      + \E\left[\begin{pmatrix} \sigma_c^2 & \rho \\ \rho & \sigma_d^2 \end{pmatrix} \otimes \calB\right], \nonumber\\
    &= \begin{pmatrix} \sigma_a^2 & 0 \\ 0 & \sigma_b^2 \end{pmatrix} \otimes \E[T\calB] 
      + \begin{pmatrix} \sigma_c^2 & \rho \\ \rho & \sigma_d^2 \end{pmatrix} \otimes \E[\calB]. \label{decompoL}
\end{align}

\noindent We shall prove that $\E[\calB]$ and $\E[T\calB]$ are both positive definite matrices. Denote by $\lambda_1$ and $\lambda_2$ the two eigenvalues of the real symmetric matrix $\E[\calB]$. We clearly have
$$\lambda_1 + \lambda_2 = tr(\E[\calB]) = \E\left[\frac{T^2+1}{(1+T)^2}\right] > 0$$
and
$$\lambda_1 \lambda_2 = \det(\E[\calB]) = \E\left[\frac{T^2}{(1+T)^2}\right] \E\left[\frac{1}{(1+T)^2}\right] - \E\left[\frac{T}{(1+T)^2}\right]^2 \geq0$$
thanks to the Cauchy-Schwarz inequality and $\lambda_1 \lambda_2 = 0$ if and only if $T$ is degenerate, which is not the case thanks to 
Lemma \ref{CVYn}. Consequently, $\E[\calB]$ is a positive definite matrix. In the same way, we can prove that $\E[T\calB]$ is also a positive definite matrix. Since the Kronecker product of two positive definite matrices is a positive definite matrix, we deduce from \eqref{decompoL} that $L$ is positive definite as soon as $\sigma_a^2>0$ and 
$\sigma_b^2>0$ or $\rho^2 < \sigma_c^2 \sigma_d^2$ which is the case thanks to \H{eps11}.


\section{Proof of Theorem \ref{CVpstheta}} \label{demoCVpstheta}


We will follow the same approach as in Bercu et al.~\cite{BercuBDSAGP}. For all $n\geq1$, let $\calV_n = M_n^t \Sigma_{n-1}^{-1} M_n = (\wh\theta_n - \theta)^t \Sigma_{n-1} (\wh \theta_n - \theta)$.
First of all, we have
\begin{align*}
 \calV_{n+1} &= M_{n+1}^t \Sigma_n^{-1} M_{n+1} = (M_n + \Delta M_{n+1})^t \Sigma_n^{-1} (M_n + \Delta M_{n+1}),\\
	      &= M_n^t \Sigma_n^{-1} M_n + 2 M_n^t \Sigma_n^{-1} \Delta M_{n+1} + \Delta M_{n+1}^t \Sigma_n^{-1} \Delta M_{n+1},\\
	      &= \calV_n - M_n^t(\Sigma_{n-1}^{-1} - \Sigma_n^{-1})M_n + 2 M_n^t \Sigma_n^{-1} \Delta M_{n+1} + \Delta M_{n+1}^t \Sigma_n^{-1} \Delta M_{n+1}.
\end{align*}

\noindent By summing over this identity, we obtain the main decomposition
\begin{equation}\label{egaliteVn}
 \calV_{n+1} + \calA_n = \calV_1 + \calB_{n+1} + \calW_{n+1}
\end{equation}

\noindent where
$$\calA_n = \sum_{k=1}^n M_k^t(\Sigma_{k-1}^{-1} - \Sigma_k^{-1})M_k,$$
$$\calB_{n+1} = 2 \sum_{k=1}^n M_k^t \Sigma_k^{-1} \Delta M_{k+1} \hspace{10pt} \text{ and } \hspace{10pt} \calW_{n+1} = \sum_{k=1}^n \Delta M_{k+1}^t \Sigma_k^{-1} \Delta M_{k+1}.$$

\begin{Lem} \label{lemCVVnAn}
 Assume that $(\veps_n)$ satisfies \H{eps1} to \H{eps8}. Then, we have
\begin{equation} \label{CVWn}
 \lim_{n\to\infty} \frac{\calW_{n}}{n} = \frac12 tr((I_2 \otimes A)^{-1/2} L (I_2 \otimes A)^{-1/2}) \hspace{20pt} \text{ a.s.}
\end{equation}
where $A$ is the positive matrix given by \eqref{defA}. In addition, we also have
\begin{equation}\label{CVBn}
\calB_{n+1} = o(n) \hspace{20pt} \text{ a.s.}
\end{equation}
and
\begin{equation} \label{CVVnAn}
 \lim_{n\to\infty} \frac{\calV_{n+1} + \calA_n}{n} = \frac12 tr((I_2 \otimes A)^{-1/2} L (I_2 \otimes A)^{-1/2}) \hspace{20pt} \text{ a.s.}
\end{equation}
\end{Lem}

\begin{proof}

\noindent First of all, we have $\calW_{n+1} = \calT_{n+1} + \calR_{n+1}$ where
$$\calT_{n+1} = \sum_{k=1}^n \frac{\Delta M_{k+1}^t (I_2 \otimes A)^{-1} \Delta M_{k+1}}{|\T_k|},$$
$$\calR_{n+1} = \sum_{k=1}^n \frac{\Delta M_{k+1}^t (|\T_k| \Sigma_k^{-1} - (I_2 \otimes A)^{-1}) \Delta M_{k+1}}{|\T_k|}.$$

\noindent One can observe that $\calT_{n+1} = tr((I_2 \otimes A)^{-1/2} \calH_{n+1} (I_2 \otimes A)^{-1/2})$ where
$$ \calH_{n+1} = \sum_{k=1}^n \frac{\Delta M_{k+1} \Delta M_{k+1}^t}{|\T_k|}.$$

\noindent Our aim is to make use of the strong law of large numbers for martingale transforms, so we start by adding and subtracting a term involving 
the conditional expectation of $\Delta \calH _{n+1}$ given $\calF_n$. We have thanks to relation \eqref{defLk} that for all 
$n\geq0$, $\E[\Delta M_{n+1} \Delta M_{n+1}^t | \calF_n] = L_n$. Consequently, we can split $\calH_{n+1}$ into two terms
$$\calH_{n+1} = \sum_{k=1}^n \frac{L_k}{|\T_k|} + \calK_{n+1},$$
where 
$$\calK_{n+1} = \sum_{k=1}^n \frac{\Delta M_{k+1} \Delta M_{k+1}^t - L_k}{|\T_k|}.$$
It clearly follows from convergence \eqref{limcrochet} that 
$$ \lim_{n\to\infty} \frac{L_n}{|\T_n|} = \frac12 L \hspace{20pt} \text{ a.s.}$$
Hence, Cesaro convergence immediately implies that
\begin{equation} \label{CesaroLk}
 \lim_{n\to\infty} \frac1n \sum_{k=1}^n \frac{L_k}{|\T_k|} = \frac12 L \hspace{20pt} \text{ a.s.}
\end{equation}
On the other hand, the sequence $(\calK_n)_{n\geq2}$ is obviously a square integrable martingale. Moreover, we have 
$$\Delta \calK_{n+1}= \calK_{n+1} - \calK_n = \frac1{|\T_n|} (\Delta M_{n+1} \Delta M_{n+1}^t - L_n).$$
For all $u\in\R^4$, denote $\calK_n(u) = u^t\calK_n u$. It follows from tedious but straightforward calculations, together with Lemma 
\ref{LFGN}, that the increasing process of the martingale $(\calK_n(u))_{n\geq2}$ satisfies $<\!\calK(u)\!>_n =\calO(n)$ a.s. Therefore, we deduce from 
the strong law of large numbers for martingales that for all $u\in\R^4$, $\calK_n(u) = o(n)$ a.s.~leading to $\calK_n = o(n)$ a.s. Hence, 
we infer from \eqref{CesaroLk} that
\begin{equation} \label{CVHn}
 \lim_{n\to\infty} \frac{\calH_{n+1}}n = \frac12 L \hspace{20pt} \text{a.s.}
\end{equation}

\noindent Via the same arguments as in the proof of convergence \eqref{limcrochet}, we find that
\begin{equation}\label{CVSigman}
\lim_{n\to\infty} \frac{\Sigma_n}{|\T_n|} = I_2 \otimes A \hspace{20pt} \text{a.s.}
\end{equation}
where $A$ is the positive definite matrix given by \eqref{defA}. Then, we obtain from \eqref{CVHn} that
$$ \lim_{n\to\infty} \frac{\calT_n}n = \frac12 tr((I_2 \otimes A)^{-1/2} L (I_2 \otimes A)^{-1/2}) \hspace{20pt} \text{a.s.}$$
which allows us to say that $\calR_n = o(n)$ a.s. leading to \eqref{CVWn} We are now in position to prove \eqref{CVBn}. Let us recall that 
$$\calB_{n+1} = 2 \sum_{k=1}^n M_k^t \Sigma_k^{-1} \Delta M_{k+1} = 2 \sum_{k=1}^n M_k^t \Sigma_k^{-1} \psi_k \xi_{k+1}.$$

\noindent Hence, $(\calB_n)_{n\geq2}$ is a square integrable martingale. In addition, we have
$$\Delta \calB_{n+1} = 2M_n^t\Sigma_n^{-1}\Delta M_{n+1}.$$

\noindent Thus
\begin{align*}
 \E[(\Delta \calB_{n+1})^2 | \calF_n] &= 4 \E[M_n^t\Sigma_n^{-1}\Delta M_{n+1} \Delta M_{n+1}^t\Sigma_n^{-1}M_n | \calF_n] \hspace{20pt} \text{a.s.} \\
	    &= 4 M_n^t\Sigma_n^{-1}\E[\Delta M_{n+1} \Delta M_{n+1}^t| \calF_n]\Sigma_n^{-1}M_n \hspace{20pt} \text{a.s.}\\
	    &= 4 M_n^t\Sigma_n^{-1} L_n \Sigma_n^{-1}M_n \hspace{20pt} \text{a.s.}
\end{align*}

\noindent We can observe that
\begin{equation*}
 L_n = \sum_{k\in\G_n} \frac1{c_k^2}\begin{pmatrix}
                        \sigma_a^2 X_k + \sigma_c^2 & \rho \\
			\rho & \sigma_b^2 X_k + \sigma_d^2
                       \end{pmatrix}
			      \otimes \begin{pmatrix}
			               X_k^2 & X_k \\ X_k & 1
			              \end{pmatrix}
\end{equation*}
and
\begin{equation*}
\psi_n\psi_n^t = \sum_{k\in\G_n} \frac{1}{c_k} I_2 \otimes \begin{pmatrix}
						X_k^2 & X_k \\ X_k & 1
					      \end{pmatrix}.
\end{equation*}
For $\alpha = \max(\sigma_a^2 +\sigma_b^2 , \sigma_c^2 + \sigma_d^2)$, denote
$$\Delta_n = \alpha c_n I_2 - \begin{pmatrix}
                        \sigma_a^2 X_n + \sigma_c^2 & \rho \\
			\rho & \sigma_b^2 X_n + \sigma_d^2
                       \end{pmatrix}.$$
It is not hard to see that $\Delta_n$ is a positive definite matrix. As a matter of fact, we deduce from the elementary inequality
\begin{equation}\label{ineg}
(\sigma_a^2+\sigma_b^2)X_n + \sigma_c+\sigma_d^2 \leq \alpha c_n
\end{equation}
that
$$tr(\Delta_n) = 2\alpha c_n - \left((\sigma_a^2+\sigma_b^2)X_n + \sigma_c^2 + \sigma_d^2 \right) \geq \alpha c_n >0.$$
In addition, we also have from \eqref{ineg} that
\begin{align*}
\det(\Delta_n) &= \left(\alpha c_n - (\sigma_a^2 X_n + \sigma_c^2)\right)\left(\alpha c_n - (\sigma_b^2 X_n +\sigma_d^2)\right) - \rho^2,\\
	&= \alpha^2 c_n^2 - \alpha c_n \left((\sigma_a^2+\sigma_b^2)X_n +\sigma_c^2 +\sigma_d^2\right)\\
	&\hspace{80pt} + (\sigma_a^2X_n + \sigma_c^2)(\sigma_b^2X_n + \sigma_d^2) - \rho^2,\\
	&\geq \sigma_a^2\sigma_b^2X_n^2 + (\sigma_a^2\sigma_d^2 + \sigma_b^2\sigma_c^2)X_n + \sigma_c^2\sigma_d^2 - \rho^2,\\
	&\geq \sigma_c^2\sigma_d^2 - \rho^2 >0
\end{align*}
thanks to \H{eps11}. Consequently,
$$\begin{pmatrix}
                        \sigma_a^2 X_n + \sigma_c^2 & \rho \\
			\rho & \sigma_b^2 X_n + \sigma_d^2
                       \end{pmatrix} \leq \alpha c_n I_2$$
which immediately implies that $L_n \leq \alpha \psi_n\psi_n^t$. Moreover, we can use Lemma B.1 of \cite{BercuBDSAGP} to say that 
$$\Sigma_{n-1}^{-1} \psi_n \psi_n^t \Sigma_n^{-1} \leq \Sigma_{n-1}^{-1} - \Sigma_n^{-1}.$$

\noindent Hence
\begin{align*}
 \E[(\Delta \calB_{n+1})^2 | \calF_n] &= 4 M_n^t\Sigma_n^{-1} L_n \Sigma_n^{-1}M_n \hspace{20pt} \text{a.s.}\\
			&\leq 4 \alpha M_n^t\Sigma_n^{-1} \psi_n\psi_n^t \Sigma_n^{-1}M_n \hspace{20pt} \text{a.s.}\\
			&\leq 4 \alpha M_n^t(\Sigma_{n-1}^{-1} - \Sigma_n^{-1})M_n \hspace{20pt} \text{a.s.}\\
\end{align*}

\noindent leading to $<\!\calB\!>_n \leq 4\alpha \calA_n$. Therefore it follows from the strong law of large numbers for martingales that $\calB_n = o(\calA_n)$. Finally, we deduce from decomposition \eqref{egaliteVn} that 
$$\calV_{n+1} + \calA_n = o(\calA_n) + \calO(n) \hspace{20pt} \text{ a.s.}$$

\noindent leading to $\calV_{n+1} = \calO(n)$ and $\calA_n = \calO(n)$ a.s.~which implies that $\calB_n= o(n)$ a.s. Finally we clearly obtain convergence \eqref{CVVnAn} from the main decomposition \eqref{egaliteVn} together with \eqref{CVWn} and \ref{CVBn}, which completes the proof of Lemma \ref{lemCVVnAn}.
\end{proof}

\begin{Lem}\label{delta}
 Assume that $(\veps_n)$ satisfies \H{eps1} to 
\H{eps8}. For all $\delta > 1/2$, we have
\begin{equation}\label{majdelta}\|M_n\|^2 = o(|\T_n|n^\delta) \hspace{20pt} \text{ a.s.}\end{equation}
\end{Lem}

\begin{proof}
Let us recall that
$$M_n = \sum_{k\in\T_{n-1}} \frac1{c_k} \begin{pmatrix} X_kV_{2k} \\ V_{2k} \\ X_kV_{2k+1} \\V_{2k+1}\end{pmatrix}.$$
Denote
$$\begin{array}{ccccc}P_n = \displaystyle \sum_{k\in\T_{n-1}} \frac{X_kV_{2k}}{c_k}  & & \text{ and } & & \displaystyle Q_n = \sum_{i\in\T_{n-1}} \frac{V_{2k}}{c_k}.\end{array}$$
On the one hand, $P_n$ can be rewritten as
$$\begin{array}{ccccc}\displaystyle{ P_n = \sum_{k=1}^n \sqrt{|\G_{k-1}|} f_k } & & \text{ where } & & \displaystyle f_n = \frac1{\sqrt{|\G_{n-1}|}}\sum_{k\in\G_{n-1}} \frac{X_kV_{2k}}{c_k}.\end{array}$$
We already saw in Section 3 that for all $k\in\G_n$,
$$\begin{array}{ccccccc} \E[V_{2k}|\calF_n] = 0 & & \text{ and } & & \E[V_{2k}^2|\calF_n] = \sigma_a^2X_k+\sigma_c^2 & & \text{a.s.} \end{array}$$
In addition, for all $k\in\G_n$, $\E[V_{2k}V_{2k+1}|\calF_n]=\rho$ and
$$\E[V_{2k}^4|\calF_n] = 3\sigma_a^4X_k^2 + X_k(\mu_a^4-3\sigma_a^4+6\sigma_a^2\sigma_c^2)+\mu_c^4 \hspace{20pt} \text{a.s.}$$
which implies that
\begin{equation}\label{majV2k4} \E[V_{2k}^4|\calF_n] \leq \mu_{ac}^4 c_k^2 \hspace{20pt} \text{a.s.}.\end{equation}
where $\mu_{ac}^4 = \mu_a^4+\mu_c^4+6\sigma_a^2\sigma_c^2$. Consequently, $\E[f_{n+1}|\calF_n]=0$ a.s.~and we deduce from \eqref{majV2k4} together with the Cauchy-Schwarz inequality that
\begin{align}
\E[f_{n+1}^4|\calF_n] &= \frac1{|\G_n|^2} \sum_{k\in\G_n} \left(\frac{X_k}{c_k}\right)^4 \E[V_{2k}^4|\calF_n]\nonumber\\
& \hspace{40pt} + \frac3{|\G_n|^2} \sum_{k\in\G_{n}}\sum_{\substack{l\in\G_{n} \\ l\neq k}}      \left(\frac{X_k}{c_k}\right)^2 \left(\frac{X_l}{c_l}\right)^2 \E[V_{2k}^2|\calF_n] \E[V_{2l}^2|\calF_n] \hspace{20pt} \text{a.s.}\nonumber\\
&\leq \frac{\mu_{ac}^4}{|\G_n|^2}(1+3\sqrt{|\G_n|(|\G_n|-1)})\sum_{k\in\G_n}c_k^2 \hspace{20pt} \text{a.s.} \nonumber\\
&\leq \frac{3\mu_{ac}^4}{|\G_n|} \sum_{k\in\G_n}c_k^2 \hspace{20pt} \text{a.s.} \label{majfn}
\end{align}
However, it follows from Lemma \ref{LFGN} that
$$\lim_{n\to\infty} \frac1{|\T_n|} \sum_{k\in\T_n} c_k^2 = \E[(1+T)^2] \hspace{20pt} \text{a.s.}$$
which is equivalent to say that
\begin{equation}\label{limck2}
\lim_{n\to\infty} \frac1{|\G_n|} \sum_{k\in\G_n} c_k^2 = \E[(1+T)^2] \hspace{20pt} \text{a.s.}
\end{equation}
Therefore, we infer from \eqref{majfn} and \eqref{limck2} that
$$\sup_{n\geq0}\E[f_{n+1}^4|\calF_n]<\infty \hspace{20pt} \text{a.s.}$$
Hence, we obtain from Wei's Lemma given in \cite{Wei} page 1672 that for all $\delta > 1/2$,
$$P_n^2 = o(|\T_{n-1}|n^\delta) \hspace{20pt} \text{a.s.}$$
On the other hand, $Q_n$ can be rewritten as
$$\begin{array}{ccccc}\displaystyle{ Q_n = \sum_{k=1}^n \sqrt{|\G_{k-1}|} g_k } & & \text{ where } & & \displaystyle g_n = \frac1{\sqrt{|\G_{n-1}|}}\sum_{k\in\G_{n-1}} \frac{V_{2k}}{c_k}.\end{array}$$
Via the same calculation as before, $\E[g_{n+1}|\calF_n] = 0$ a.s.~and, as $c_n\geq1$,
$$\E[g_{n+1}^4|\calF_n]\leq \frac{3\mu_{bd}^4}{|\G_n|} \sum_{k\in\G_n}\frac1{c_k^2}\leq 3\mu_{bd}^4 \hspace{20pt} \text{a.s.}$$
Hence, we deduce once again from Wei's Lemma that for all $\delta > 1/2$,
$$Q_n^2 = o(|\T_{n-1}|n^\delta) \hspace{20pt} \text{a.s.}$$
In the same way, we obtain the same result for the two last components of $M_n$ which completes the proof of Lemma \ref{delta}.
\end{proof}

\noindent {\bf Proof of Theorem \ref{CVpstheta}.} We recall from \eqref{difftheta} that $\wh \theta_n-\theta = \Sigma_{n-1}^{-1} M_n$ which implies
$$\|\wh \theta_n - \theta \|^2 \leq \frac{\calV_n}{\lambda_{min}(\Sigma_{n-1})}$$
where $\calV_n = M_n^t \Sigma_{n-1}^{-1} M_n$. On the one hand, it follows from \eqref{CVVnAn} that $\calV_n=\calO(n)$ a.s. On the other hand, we deduce from \eqref{CVSigman} that
$$\lim_{n\to\infty} \frac{\lambda_{min}(\Sigma_n)}{|\T_n|} = \lambda_{min}(A) >0 \hspace{20pt} \text{a.s.}$$
Consequently, we find that
$$\|\wh \theta_n - \theta \|^2 = \calO\left(\frac{n}{|\T_{n-1}|}\right) \hspace{20pt} \text{a.s.}$$

\noindent We are now in position to prove the quadratic strong law \eqref{quadratic1}. First of all a direct application of Lemma \ref{delta} ensures that $\calV_n = o(n^\delta)$ a.s.~for all $\delta > 1/2$. Hence, we obtain from \eqref{CVVnAn} that
\begin{equation}\lim_{n\to\infty} \frac{\calA_n}n = \frac12 tr((I_2 \otimes A)^{-1/2} L (I_2 \otimes A)^{-1/2}) \hspace{20pt} \text{ a.s.}\label{CVAn/n}\end{equation}

\noindent Let us rewrite $\calA_n$ as 
$$\calA_n = \sum_{k=1}^n M_k^t\left(\Sigma_{k-1}^{-1} - \Sigma_k^{-1}\right) M_k = \sum_{k=1}^n M_k^t\Sigma_{k-1}^{-1/2} \Delta_k \Sigma_{k-1}^{-1/2} M_k$$

\noindent where $\Delta_k = I_4 - \Sigma_{k-1}^{1/2}\Sigma_{k}^{-1}\Sigma_{k-1}^{1/2}$. We already saw from \eqref{CVSigman} that 
$$\lim_{n\to\infty} \frac{\Sigma_n}{|\T_n|} = I_2\otimes A \hspace{20pt} \text{a.s.}$$
which ensures that
$$\displaystyle \lim_{n\to\infty} \Delta_n = \frac12 I_4 \hspace{20pt} \text{a.s.}$$
In addition, we deduce from \eqref{CVVnAn} that $\calA_n=\calO(n)$ a.s.~which implies that
\begin{equation}\frac{\calA_n}n = \left(\frac1{2n}\sum_{k=1}^nM_k^t\Sigma_{k-1}^{-1}M_k\right) + o(1) \hspace{20pt} \text{ a.s.}\label{eqAn/n}\end{equation}

\noindent Moreover we have
\begin{align}
 \frac1n\sum_{k=1}^nM_k^t\Sigma_{k-1}^{-1}M_k &= \frac1n \sum_{k=1}^n (\wh \theta_k - \theta)^t\Sigma_{k-1}(\wh \theta_k - \theta), \nonumber\\
	  &= \frac1n \sum_{k=1}^n |\T_{k-1}| (\wh \theta_k - \theta)^t\frac{\Sigma_{k-1}}{|\T_{k-1}|}(\wh \theta_k - \theta), \nonumber\\
	  &= \frac1n \sum_{k=1}^n |\T_{k-1}| (\wh \theta_k - \theta)^t(I_2 \otimes A)(\wh \theta_k - \theta) + o(1) \hspace{20pt} \text{ a.s.} \label{elementQSL}
\end{align}

\noindent Therefore, \eqref{CVAn/n} together with \eqref{eqAn/n} and \eqref{elementQSL} lead to \eqref{quadratic1}. 

\section{Proof of Theorem \ref{CVpsvar}} \label{demoCVpsvar}

First of all, we shall only prove \eqref{vitesseeta} since the proof of \eqref{vitesseetad} follows exactly the same lines. We clearly have from \eqref{esteta} that
\begin{align} 
Q_{n-1}(\wh \eta_{n} - \eta_{n}) &= \sum_{k\in\T_{n-1}} \frac1{d_k}(\wh V_{2k}^2 - V_{2k}^2) \Phi_k,\nonumber\\
	&=\sum_{l=0}^{n-1}\sum_{k\in\G_{l}} \frac1{d_k}(\wh V_{2k}^2 - V_{2k}^2) \Phi_k,\nonumber\\
	&= \sum_{l=0}^{n-1} \sum_{k\in\G_l} \frac1{d_k}\left((\wh V_{2k} - V_{2k})^2 + 2(\wh V_{2k} - V_{2k})V_{2k}\right) \Phi_k.\label{diffeta}
\end{align}
In addition, we already saw in Section 3 that for all $l\geq0$ and $k\in\G_l$,
$$\wh V_{2k} - V_{2k} = -\begin{pmatrix} \wh a_l-a \\ \wh c_l - c \end{pmatrix}^t \Phi_k.$$
Consequently,
$$(\wh V_{2k} - V_{2k})^2 \leq \|\Phi_k\|^2 \left((\wh a_l-a )^2 + (\wh c_l-c )^2\right).$$
Hence, we obtain that
\begin{align}
\left\|\sum_{l=0}^{n-1} \sum_{k\in\G_l} \frac{(\wh V_{2k} - V_{2k})^2}{d_k} \Phi_k\right\| &\leq \sum_{l=0}^{n-1} \sum_{k\in\G_l} \frac{\|\Phi_k\|^3}{d_k} \left((\wh a_l-a )^2 + (\wh c_l-c )^2\right),\nonumber\\
	&\leq \sum_{l=0}^{n-1} \left((\wh a_l-a )^2 + (\wh c_l-c )^2\right) \sum_{k\in\G_l} c_k,\nonumber\\
	&\leq \sum_{l=0}^{n-1} \left((\wh a_l-a )^2 + (\wh c_l-c )^2\right) |\T_{l-1}| \frac1{|\T_{l-1}|}\sum_{k\in\G_l} c_k. \label{majterme1}
\end{align}
Moreover, we can deduce from Lemma \ref{LFGN} that 
\begin{equation}
\lim_{n\to\infty} \frac1{|\T_{n-1}|}\sum_{k\in\G_n}c_k = \E[1+T] \hspace{20pt} \text{a.s.} \label{limck}
\end{equation}
Then, we find from \eqref{majterme1} and \eqref{limck} that 
$$\left\|\sum_{l=0}^{n-1} \sum_{k\in\G_l} \frac{(\wh V_{2k} - V_{2k})^2}{d_k} \Phi_k\right\| = \calO\left(\sum_{l=0}^{n-1} |\T_{l-1}| \left((\wh a_l-a )^2 + (\wh c_l-c )^2\right)\right) \hspace{20pt} \text{a.s.}$$
However, as $\Lambda$ is positive definite, we obtain from \eqref{quadratic1} that 
$$\sum_{l=0}^{n-1} |\T_{l-1}| \left((\wh a_l-a )^2 + (\wh c_l-c )^2\right) = \calO(n) \hspace{20pt} \text{a.s.}$$
which implies that 
\begin{equation}
\left\|\sum_{l=0}^{n-1} \sum_{k\in\G_l} \frac{(\wh V_{2k} - V_{2k})^2}{d_k} \Phi_k\right\| = \calO(n) \hspace{20pt} \text{a.s.} \label{Oterme1}
\end{equation}
Furthermore, denote
$$P_n = \sum_{l=0}^{n-1} \sum_{k\in\G_l} \frac{(\wh V_{2k} - V_{2k})V_{2k}}{d_k} \Phi_k.$$
We clearly have
\begin{align*}
\Delta P_{n+1} &= P_{n+1} - P_n = \sum_{k\in\G_n} \frac{(\wh V_{2k} - V_{2k})V_{2k}}{d_k} \Phi_k,\\
	&=-\sum_{k\in\G_n} \frac{V_{2k}}{d_k} \Phi_k \Phi_k^t \begin{pmatrix} \wh a_l-a \\ \wh c_l - c \end{pmatrix}.
\end{align*}
In addition, for all $k\in\G_n$, $\E[V_{2k}|\calF_n] = 0$ a.s.~and $\E[V_{2k}^2|\calF_n] = \sigma_a^2 X_k +\sigma_c^2 \leq \alpha c_k$ a.s.~where $\alpha=\max(\sigma_a^2,\sigma_c^2)$. Consequently, $\E[\Delta P_{n+1} |\calF_n] = 0$ a.s.~and
\begin{align*}
\E[\Delta P_{n+1} \Delta P_{n+1}^t |\calF_n] &= \sum_{k\in\G_n}\frac1{d_k^2}\E[V_{2k}^2|\calF_n] \Phi_k\Phi_k^t \begin{pmatrix} \wh a_l-a \\ \wh c_l - c \end{pmatrix} \begin{pmatrix} \wh a_l-a \\ \wh c_l - c \end{pmatrix}^t \Phi_k \Phi_k^t \hspace{20pt} \text{a.s.}\\
	&= \sum_{k\in\G_n}\frac{\sigma_a^2 X_k +\sigma_c^2}{d_k^2} \Phi_k\Phi_k^t \begin{pmatrix} \wh a_l-a \\ \wh c_l - c \end{pmatrix} \begin{pmatrix} \wh a_l-a \\ \wh c_l - c \end{pmatrix}^t \Phi_k \Phi_k^t \hspace{20pt} \text{a.s.}
\end{align*}
Therefore, $(P_n)$ is a square integrable vector martingale with increasing process $<\!P\!>_n$ given by
\begin{align*}
<\!P\!>_n &= \sum_{l=1}^{n-1} \E[\Delta P_{l+1} \Delta P_{l+1}^t |\calF_l] \hspace{20pt} \text{a.s.}\\
	&= \sum_{l=1}^{n-1} \sum_{k\in\G_l}\frac{\sigma_a^2 X_k +\sigma_c^2}{d_k^2} \Phi_k\Phi_k^t \begin{pmatrix} \wh a_l-a \\ \wh c_l - c \end{pmatrix} \begin{pmatrix} \wh a_l-a \\ \wh c_l - c \end{pmatrix}^t \Phi_k \Phi_k^t \hspace{20pt} \text{a.s.}
\end{align*}
It immediately follows from the previous calculation that
\begin{align*}
\|<\!P\!>_n\| &\leq \alpha \sum_{l=0}^{n-1} \left((\wh a_l-a )^2 + (\wh c_l-c )^2\right)\sum_{k\in\G_l}\frac{\|\Phi_k\|^4 c_k}{d_k^2} \hspace{20pt} \text{a.s.}\\
&\leq \alpha \sum_{l=0}^{n-1} \left((\wh a_l-a )^2 + (\wh c_l-c )^2\right)\sum_{k\in\G_l}c_k \hspace{20pt} \text{a.s.}
\end{align*}
leading to
$$\|<\!P\!>_n\| = \calO(n) \hspace{20pt} \text{a.s.}$$
Then, we deduce from the strong law of large numbers for martingale given e.g.~in Theorem 1.3.15 of \cite{Duflo} that
\begin{equation}
P_n = o(n) \hspace{20pt} \text{a.s.} \label{oPn}
\end{equation}
Hence, we find from \eqref{diffeta}, \eqref{Oterme1} and \eqref{oPn} that
$$\|Q_{n-1}(\wh \eta_{n} - \eta_{n})\| = \calO(n) \hspace{20pt} \text{a.s.}$$
Moreover, we infer once again from Lemma \ref{LFGN} that
\begin{equation}
\lim_{n\to\infty} \frac1{|\T_n|}Q_n = \E\left[\begin{pmatrix} \frac{T^2}{(1+T)^2} & \frac{T}{(1+T)^2} \\ \frac{T}{(1+T)^2} & \frac{1}{(1+T)^2} \end{pmatrix}\right] \hspace{20pt} \text{ a.s.} \label{CVQ}
\end{equation}
which ensures that
$$\|\wh \eta_{n} - \eta_{n}\| = \calO\left(\frac n{|\T_{n-1}|}\right) \hspace{20pt} \text{a.s.}$$

\noindent It remains to establish \eqref{vitesserho}. Denote
$$\begin{array}{ccccc}
\wh W_n = \begin{pmatrix} \wh V_{2n}\\ \wh V_{2n+1} \end{pmatrix} & & \text{ and } & & R_n = \displaystyle\sum_{k\in\T_{n-1}} \left( \wh W_k - W_k \right)^t J W_k
\end{array}$$
where
$$J = \begin{pmatrix} 0 & 1 \\ 1 & 0 \end{pmatrix}.$$
Then, we have
$$|\T_{n-1}| (\wh \rho_n - \rho_n) = \sum_{k\in\T_{n-1}} \left(\wh V_{2k} - V_{2k} \right)\left(\wh V_{2k+1} - V_{2k+1} \right)  + R_n.$$
 It is not hard to see that $(R_n)$ is a square integrable real martingale with increasing process given by
\begin{align*}
<\!R\!>_n &= \sum_{l=0}^{n-1} \sum_{k\in\G_l} \E\left[\left. (\wh W_k - W_k)^t J W_kW_k^t J (\wh W_k - W_k) \right|\calF_n \right] \hspace{20pt} \text{a.s.}\\
	  &= \sum_{l=0}^{n-1} \sum_{k\in\G_l} (\wh W_k - W_k)^t J \E\left[\left. W_kW_k^t\right|\calF_n \right] J (\wh W_k - W_k) \hspace{20pt} \text{a.s.}\\
	  &= \sum_{l=0}^{n-1} \sum_{k\in\G_l} (\wh W_k - W_k)^t J \begin{pmatrix} \sigma_a^2 X_k + \sigma_c^2 & \rho \\ \rho & \sigma_b^2 X_k + \sigma_d^2 \end{pmatrix} J (\wh W_k - W_k) \hspace{20pt} \text{a.s.}\\
	  &= \sum_{l=0}^{n-1} \sum_{k\in\G_l} (\wh W_k - W_k)^t \begin{pmatrix} \sigma_b^2 X_k + \sigma_d^2 & \rho \\ \rho & \sigma_a^2 X_k + \sigma_c^2 \end{pmatrix} (\wh W_k - W_k) \hspace{20pt} \text{a.s.}
\end{align*}
Consequently,
\begin{align*}
<\!R\!>_n &\leq \sum_{l=0}^{n-1} \sum_{k\in\G_l} \left((\sigma_a^2 +\sigma_b^2) X_k + \sigma_c^2 +\sigma_d^2 \right) \|\wh W_k - W_k\|^2 \hspace{20pt} \text{a.s.}\\
	&\leq 2\beta \sum_{l=0}^{n-1} \left((\wh a_l-a )^2 + (\wh b_l-b )^2\right) \sum_{k\in\G_l}X_k^2 c_k\\
 	& \hspace{40pt} + 2\beta \sum_{l=0}^{n-1} \left((\wh c_l-c )^2 + (\wh d_l-d )^2\right) \sum_{k\in\G_l} c_k \hspace{20pt} \text{a.s.}
\end{align*}
where $\beta=\max(\sigma_a^2+\sigma_b^2,\sigma_c^2+\sigma_d^2)$. As previously, we obtain through Lemma \ref{LFGN} together with \eqref{quadratic1} that $<\!R\!>_n=\calO(n)$ a.s.~which ensures that $R_n=o(n)$ a.s. Moreover,
\begin{align*}
\left|\sum_{k\in\T_{n-1}} \left(\wh V_{2k} - V_{2k} \right)\left(\wh V_{2k+1} - V_{2k+1} \right)\right| &\leq \frac12 \sum_{k\in\T_{n-1}} \left(\left(\wh V_{2k} - V_{2k} \right)^2+\left(\wh V_{2k+1} - V_{2k+1} \right)^2\right),\\
	&\leq \frac12 \sum_{l=0}^{n-1} \|\wh \theta_l - \theta\|^2 \sum_{k\in\G_l} (1+X_k^2)
\end{align*}
which implies via Lemma \ref{LFGN} and \eqref{quadratic1} that
$$\sum_{k\in\T_{n-1}} \left(\wh V_{2k} - V_{2k} \right)\left(\wh V_{2k+1} - V_{2k+1} \right) = \calO(n) \hspace{20pt} \text{a.s.}$$
Therefore, we obtain that
$$|\T_{n-1}| (\wh \rho_n - \rho_n) = \calO(n) \hspace{20pt} \text{a.s.}$$
which leads to \eqref{vitesserho}.
Finally, it only remains to prove the a.s.~convergence of $\eta_n$, $\zeta_n$ and $\rho_n$ to $\eta$, $\zeta$ and $\rho$ which will immediately lead to the a.s.~convergence of $\wh \eta_n$, $\wh \zeta_n$ and $\wh \rho_n$ through \eqref{vitesseeta}, \eqref{vitesseetad} and \eqref{vitesserho}, respectively. On the one hand,
\begin{equation}
Q_{n-1} (\eta_n-\eta) = N_n = \sum_{k\in\T_n} \frac1{d_k} \Phi_k v_{2k} \label{diffetan}
\end{equation}
where we recall that $v_{2n} = V_{2n}^2 - \eta^t\Phi_n$. It is clear that $(N_n)$ is a square integrable vector martingale with increasing process $<\!N\!>_n$ given by
$$<\!N\!>_n = \sum_{l=0}^{n-1} \sum_{k\in\G_l} \frac1{d_k^2} \Phi_k \Phi_k^t (2 \sigma_a^4 X_k^2 + (\mu_a^4-3\sigma_a^4+4\sigma_a^2\sigma_c^2)X_k + \mu_c^4 - \sigma_c^4) \hspace{20pt} \text{a.s.}$$
Hence,
$$<\!N\!>_n \leq \gamma \sum_{l=0}^{n-1} \sum_{k\in\G_l} \frac1{d_k} \Phi_k \Phi_k^t \hspace{20pt} \text{a.s.}$$
where $\gamma = \mu_a^4 -\sigma_a^4 + 4\sigma_a^2\sigma_c^2 + \mu_c^4 -\sigma_c^4$, which implies that
$$\|<\!N\!>_n\| = \calO(|\T_{n-1}|) \hspace{20pt} \text{a.s.}$$
Consequently,
$$\|N_n\|^2 = \calO(n|\T_{n-1}|) \hspace{20pt} \text{a.s.}$$
which leads via \eqref{CVQ} and \eqref{diffetan} to the a.s.~convergence of $\eta_n$ to $\eta$ and to the rate of convergence of Remark \ref{remrate}. The proof of the a.s.~convergence of $\zeta_n$ to $\zeta$ follows exactly the same lines. On the other hand
\begin{equation}
|\T_{n-1}|(\rho_n-\rho) = H_n = \sum_{k\in\T_{n-1}}(V_{2k}V_{2k+1} - \rho) \label{diffrhon}
\end{equation}
It is obvious to see that $(H_n)$ is a square integrable real martingale with increasing process $<\!H\!>_n$ such that $<\!H\!>_n= \calO(|\T_{n-1}|)$ a.s. Finally, as $H_n^2 = \calO(n|\T_{n-1}|)$ a.s., we deduce from \eqref{diffrhon} that $\rho_n$ goes a.s.~to $\rho$ and that the rate of convergence of Remark \ref{remrate} is verified, which completes the proof of Theorem \ref{CVpsvar}.

\section{Proof of Theorem \ref{TCL}}\label{demoTCL}


In order to establish the asymptotic normality of our estimators, we will extensively make use of the central limit theorem for triangular arrays of vector martingales given e.g.~by Theorem 2.1.9 of \cite{Duflo}. First of all, instead of using the generation-wise filtration $(\calF_n)$, we will use the sister pair-wise filtration $(\calG_n)$ given by 
$$\calG_n = \sigma(X_1,(X_{2k},X_{2k+1}),1\leq k \leq n ).$$

\noindent {\bf Proof of Theorem \ref{TCL}, first part.} We focus our attention to the proof of the asymptotic normality \eqref{TCLtheta}. Let $M^{(n)} = (M_k^{(n)})$ be the square integrable vector martingale defined as
\begin{equation}
M_k^{(n)} = \frac1{\sqrt{|\T_n|}}\sum_{i=1}^k D_i \label{defnewmart}
\end{equation}
where
$$D_i = \frac1{c_i}\begin{pmatrix}X_iV_{2i}\\V_{2i}\\X_iV_{2i+1}\\V_{2i+1}\end{pmatrix}.$$
We clearly have
\begin{equation}
M_{t_n}^{(n)} = \frac1{\sqrt{|\T_n|}}\sum_{i=1}^{t_n} D_i = \frac1{\sqrt{|\T_n|}} M_{n+1} \label{liennewmart}
\end{equation}
where $t_n=|\T_n|$. Moreover, the increasing process associated to $(M_k^{(n)})$ is given by
\begin{align*}
<\!M^{(n)}\!>_{k} &= \frac1{|\T_n|} \sum_{i=1}^k \E\left[D_iD_i^t|\calG_{i-1}\right],\\
&= \frac1{|\T_n|} \sum_{i=1}^k \frac1{c_i^2} \begin{pmatrix}
                                \sigma_a^2X_i+\sigma_c^2 & \rho\\\rho & \sigma_b^2X_i+\sigma_d^2\end{pmatrix} \otimes \begin{pmatrix}X_i^2 & X_i\\X_i & 1 \end{pmatrix} \hspace{20pt} \text{a.s.}
\end{align*}
Consequently, it follows from convergence \eqref{limcrochet} that
$$\lim_{n\to\infty} <\!M^{(n)}\!>_{t_n} = L \hspace{20pt} \text{a.s.}$$
It is now necessary to verify Lindeberg's condition by use of Lyapunov's condition. Denote
$$\phi_n = \sum_{k=1}^{t_n} \E\left[\left.\|M_k^{(n)} - M_{k-1}^{(n)} \|^4 \right| \calG_{k-1}\right].$$
We obtain from \eqref{defnewmart} that
\begin{align*}
\phi_n &= \frac1{|\T_n|^2} \sum_{k=1}^{t_n} \E\left[\left.\frac{(1+X_k^2)^2}{(c_k)^4}(V_{2k}^2 + V_{2k+1}^2)^2\right|\calG_{k-1}\right],\\
	&\leq \frac2{|\T_n|^2} \sum_{k=1}^{t_n} \left(\E[V_{2k}^4|\calG_{k-1}] + \E[V_{2k+1}^4|\calG_{k-1}]\right).
\end{align*}
In addition, we already saw in Section \ref{demoCVpstheta} that
$$\E[V_{2n}^4|\calG_{n-1}] \leq \mu_{ac}^4 c_n^2, ~~~~~ \E[V_{2n+1}^4|\calG_{n-1}] \leq \mu_{bd}^4 c_n^2 \hspace{20pt} \text{a.s.}$$
where $\mu_{ac}^4 = \mu_a^4+\mu_c^4 + 6 \sigma_a^2\sigma_c^2$ and $\mu_{bd}^4 = \mu_b^4+\mu_d^4 + 6 \sigma_b^2\sigma_d^2$. Hence,
$$\phi_n \leq \frac{2\mu^4}{|\T_n|^2} \sum_{k=1}^{t_n} c_k^2 \hspace{20pt} \text{a.s.}$$
where $\mu^4 = \mu_{ac}^4 + \mu_{bd}^4$. We can deduce from Lemma \ref{LFGN} that
$$\lim_{n\to\infty} \frac1{|\T_n|} \sum_{k\in\T_n} c_k^2 = \E[(1+T)^2] \hspace{20pt} \text{a.s.}$$
which immediately implies that
$$\lim_{n\to\infty} \phi_n = 0 \hspace{20pt} \text{a.s.}$$
Therefore, Lyapunov's condition is satisfied and Theorem 2.1.9 of \cite{Duflo} allows us to say via \eqref{liennewmart} that
$$\frac1{\sqrt{|\T_{n-1}|}} M_n \liml \calN(0,L).$$
Finally, we infer from \eqref{difftheta} together with \eqref{CVSigman} and Slutsky's lemma that
$$\phantom{\square} \hspace{74pt} \sqrt{|\T_{n-1}|}(\wh\theta_n - \theta) \liml \calN(0,(I_2\otimes A^{-1})L(I_2\otimes A^{-1})).\hspace{74pt} \square$$

{\bf Proof of Theorem \ref{TCL}, second part.} We shall now establish the asymptotic normality given by \eqref{TCLeta}. Denote by $N^{(n)} = (N_k^{(n)})$ the square integrable vector martingale defined as
\begin{equation*}
N_k^{(n)} = \frac1{\sqrt{|\T_n|}}\sum_{i=1}^k \frac{v_{2i}}{d_i} \Phi_i.
\end{equation*}
We immediately see from \eqref{diffetan} that
\begin{equation}N_{t_n}^{(n)} = \frac1{\sqrt{|\T_n|}} Q_n(\eta_{n+1}-\eta) = \frac1{\sqrt{|\T_n|}} N_{n+1}.\label{lienN}\end{equation}
In addition, the increasing process associated to $(N_k^{(n)})$ is given by
\begin{align*}
<\!N^{(n)}\!>_k &= \frac1{|\T_n|} \sum_{i=1}^k \E\left[\left.\frac{v_{2i}^2}{d_i^2} \Phi_i\Phi_i^t \right| \calG_{i-1}\right],\\
	&= \frac1{|\T_n|} \sum_{i-1}^k \frac1{d_i^2} \Phi_i \Phi_i^t (2 \sigma_a^4 X_i^2 + (\mu_a^4-3\sigma_a^4+4\sigma_a^2\sigma_c^2)X_i + \mu_c^4 - \sigma_c^4) \hspace{20pt} \text{a.s.}
\end{align*}
Consequently, we obtain from Lemma \ref{LFGN} that
$$\lim_{n\to\infty} <\!N^{(n)}\!>_{t_n} = \E\left[\frac{2 \sigma_a^4 T^2 + (\mu_a^4-3\sigma_a^4+4\sigma_a^2\sigma_c^2)T + (\mu_c^4 - \sigma_c^4)}{(1+T)^4} \begin{pmatrix} T^2 & T \\ T & 1 \end{pmatrix} \right]=M_{ac} \hspace{10pt} \text{a.s.}$$
In order to verify Lyapunov's condition, let
$$\phi_n = \sum_{k=1}^{t_n} \E\left[\left.\|N_k^{(n)} - N_{k-1}^{(n)}\|^3 \right| \calG_{k-1}\right].$$
We clearly have
$$\|N_k^{(n)} - N_{k-1}^{(n)}\|^2 = \frac1{|\T_n|} \frac{(1+X_k^2)v_{2k}^2}{d_k^2} \leq \frac1{|\T_n|} \frac{v_{2k}^2}{d_k},$$
which implies that
$$\|N_k^{(n)} - N_{k-1}^{(n)}\|^3 \leq \frac1{|\T_n|^{3/2}} \frac{|v_{2k}|^3}{d_k^{3/2}}.$$
However,
\begin{align}
|v_{2k}|^3 &= |V_{2k}^2 - \sigma_a^2X_k - \sigma_c^2|^3 \leq (V_{2k}^2 + \sigma_a^2X_k + \sigma_c^2)^3 \nonumber\\
	&\leq V_{2k}^6 + 3V_{2k}^4(\sigma_a^2X_k + \sigma_c^2) + 3V_{2k}^2(\sigma_a^2X_k + \sigma_c^2)^2 + (\sigma_a^2X_k + \sigma_c^2)^3 \label{majv2k3}
\end{align}
We already saw that $\E[V_{2k}^2|\calG_{k-1}] = \sigma_a^2 X_k + \sigma_c^2$ a.s.~and it follows from \eqref{majV2k4} that
$$\E[V_{2k}^4|\calG_{k-1}]\leq \mu_{ac} c_k^2 \hspace{20pt} \text{a.s.}$$
It only remains to study $\E[V_{2k}^6|\calG_{k-1}]$. Denote
$$\begin{array}{ccccc}
A_k=\displaystyle\sum_{i=1}^{X_k} (Y_{k,i}-a) & & \text{ and } & & B_k=\veps_{2k}-c.
\end{array}$$
We clearly have from the identity $V_{2k} = A_k+B_k$ that
\begin{multline}\label{EV2k6}
\E[V_{2k}^6|\calG_{k-1}] = \E[A_k^6|\calG_{k-1}] + 15\E[A_k^4|\calG_{k-1}]\E[B_k^2|\calG_{k-1}]\\ + 20 \E[A_k^3|\calG_{k-1}]\E[B_k^3|\calG_{k-1}] + \E[A_k^2|\calG_{k-1}]\E[B_k^4|\calG_{k-1}] + \E[B_k^6|\calG_{k-1}].
\end{multline}
On the one hand, $\E[A_k^2|\calG_{k-1}]= \sigma_a^2 X_k$ a.s.~and
$$\E[A_k^4|\calG_{k-1}] = \mu_a^4 X_k + 3X_k(X_k-1)\sigma_a^4 \hspace{20pt} \text{a.s.}$$
Moreover, we have from Cauchy-Schwarz inequality that
$$\left|\E[A_k^3|\calG_{k-1}]\right| \leq \mu_a^2 \sigma_a X_k \hspace{20pt} \text{a.s.}$$
Furthermore, it follows from tedious but straightforward calculations that
\begin{multline*}
\E[A_k^6|\calG_{k-1}] \leq \tau_a^6X_k + 15 X_k (X_k-1) \mu_a^4\sigma_a^2 + 15\sigma_a^6 X_k(X_k-1)(X_k-2) \\+10\mu_a^6X_k(X_k-1) \hspace{20pt} \text{a.s.}
\end{multline*}
Then, it exists some constant $\alpha>0$ such that
$$\E[A_k^6|\calG_{k-1}] \leq \alpha c_k^3 \hspace{20pt} \text{a.s.}$$
On the other hand, $\E[B_k^2|\calG_{k-1}]=\sigma_c^2$ a.s.~and $\E[B_k^4|\calG_{k-1}]=\mu_c^4$ a.s. In addition
$$\begin{array}{ccccccc}
\big|\E[B_k^3|\calG_{k-1}]\big| \leq \mu_c^2\sigma_c & & \text{and} & & \E[B_k^6|\calG_{k-1}] \leq \tau_c^6 & & \text{a.s.}
\end{array}$$
Consequently, we deduce from \eqref{EV2k6} that it exists some constant $\beta>0$ such that
$$\E[V_{2k}^6|\calG_{k-1}] \leq \beta c_k^3 \hspace{20pt} \text{a.s.}$$
which implies from \eqref{majv2k3} that for some constant $\gamma>0$,
$$\E[|v_{2k}|^3|\calG_{k-1}] \leq \gamma c_k^3 \hspace{20pt} \text{a.s.}$$
Then, as $c_k^2=d_k$, we can conclude that
$$\phi_n \leq \frac\gamma{\sqrt{|\T_n|}} \hspace{20pt} \text{a.s.}$$
which immediately leads to
$$\lim_{n\to\infty} \phi_n = 0 \hspace{20pt} \text{a.s.}$$
Therefore, Lyapunov's condition is satisfied and we find from Theorem 2.1.9 of \cite{Duflo} and \eqref{lienN} that
\begin{equation}\label{limN}
\frac1{\sqrt{|\T_{n-1}|}}N_n \liml \calN(0,M_{ac}).
\end{equation}
Hence, we obtain from \eqref{CVQ}, \eqref{limN} and Slutsky's lemma that
$$\sqrt{|\T_{n-1}|}(\eta_n-\eta) \liml \calN(0,B^{-1}M_{ac} B^{-1}).$$
Finally, \eqref{vitesseeta} ensures that
$$\sqrt{|\T_{n-1}|}(\wh\eta_n-\eta) \liml \calN(0,B^{-1}M_{ac} B^{-1}).$$
The proof of \eqref{TCLetad} follows exactly the same lines. \hfill $\square$

{\bf Proof of Theorem \ref{TCL}, third part.} It remains to establish the asymptotic normality given by \eqref{TCLrho}. Denote by $H^{(n)} = (H_k^{(n)})$ the square integrable martingale defined as
\begin{equation} \label{lienH}
H_k^{(n)} = \frac1{\sqrt{|\T_n|}} \sum_{i=1}^k (V_{2i}V_{2i+1} - \rho).
\end{equation}
We clearly have from \eqref{diffrhon} that
$$H_{t_n}^{(n)} = \sqrt{|\T_n|} (\rho_{n+1} -\rho) = \frac1{\sqrt{|\T_n|}}H_{n+1}.$$
Moreover, the increasing process of $(H_k^{(n)})$ is given by
$$<\!H^{(n)}\!>_k = \frac1{|\T_n|} \sum_{i=1}^k\left(\E[V_{2i}^2V_{2i+1}^2|\calG_{n-1}] - \rho^2\right).$$
As before, let
$$\begin{array}{ccccc}
C_k=\displaystyle\sum_{i=1}^{X_k} (Z_{k,i}-b) & & \text{ and } & & B_k=\veps_{2k+1}-d.
\end{array}$$
As $V_{2k}=A_k+B_k$ and $V_{2k+1}=C_k+D_k$, we clearly have
\begin{multline*} \E\left[\left.V_{2k}^2V_{2k+1}^2\right| \calG_{k-1}\right] = \E\left[\left. A_k^2 \right| \calG_{k-1}\right] \left(\E\left[\left. C_k^2 \right| \calG_{k-1}\right] + \E\left[\left. D_k^2 \right| \calG_{k-1}\right]\right) \\
	     + \E\left[\left. B_k^2 \right| \calG_{k-1}\right]\E\left[\left. C_k^2 \right| \calG_{k-1}\right]
		+ \E\left[\left. B_k^2D_k^2 \right| \calG_{k-1}\right] \hspace{20pt} \text{a.s.}
\end{multline*}
Consequently,
\begin{equation}\label{EV2V2}
\E\left[\left.V_{2k}^2V_{2k+1}^2\right| \calG_{k-1}\right] = \sigma_a^2\sigma_b^2X_k^2 + \left(\sigma_a^2\sigma_d^2 + \sigma_b^2\sigma_c^2\right)X_k + \nu^2 \hspace{20pt} \text{a.s.}
\end{equation}
Then, we deduce once again from Lemma \ref{LFGN} that
$$\lim_{n\to\infty} <\!H^{(n)}\!>_{t_n} =\sigma_{\rho}^2 \hspace{20pt} \text{a.s.}$$
where $\sigma_\rho^2$ is given by \eqref{sigmarho}. One can observe that
\begin{align*}
\E[T^2] &= \E\left[\left(\sum_{k=2}^\infty a_2 \circ \hdots \circ a_{k-1} \circ e_k\right)^2 \right],\\
	&= \sum_{k=2}^\infty \E\left[\left( a_2 \circ \hdots \circ a_{k-1} \circ e_k\right)^2 \right]\\
	&\hspace{90pt} + \sum_{k=2}^\infty \sum_{\substack{l=2 \\ l\neq k}}^\infty \E\left[ a_2 \circ \hdots \circ a_{k-1} \circ e_k \right]\E\left[ a_2 \circ \hdots \circ a_{l-1} \circ e_l \right].
\end{align*}
Moreover, we have thanks to calculations of Section \ref{preuve lemme CVYn}  and \ref{demoLFGN} that
$$\E\left[\left( a_2 \circ \hdots \circ a_{k-1} \circ e_k\right)^2 \right] = \Upsilon\overline{c}\left(\overline{a}^{k-2} - \overline{a^2}^{k-2}\right) + \overline{a^2}^{k-2}\overline{c^2},$$
$$\E\left[ a_2 \circ \hdots \circ a_{k-1} \circ e_k \right] =\overline{a}^{k-2} \overline{c}.$$
Hence
\begin{align*}
\E[T^2] &= \frac{\Upsilon\overline{c}}{1-\overline{a}} + \frac{\overline{c^2} - \Upsilon \overline{c}}{1-\overline{a^2}} + \overline{c^2}\left( \frac1{(1-\overline{a})^2} - \frac1{1-\overline{a}^2} \right),\\
	&= \frac{\Upsilon\overline{c}}{1-\overline{a}} + \frac{\overline{c^2} - \Upsilon \overline{c}}{1-\overline{a^2}} + \frac{2\overline{a}(\overline{c}^2)}{(1-\overline{a})(1-\overline{a}^2)}.
\end{align*}
In order to verify Lyapunov's condition, denote
$$\phi_n = \sum_{k=1}^{t_n} \E\left[\left.|H_k^{(n)} - H_{k-1}^{(n)}|^3\right|\calG_{k-1}\right].$$
We obtain from \eqref{lienH} that
\begin{align}
\phi_n &= \frac1{|\T_n|^{3/2}} \sum_{k=1}^{t_n} \E\left[\left.|V_{2k}V_{2k+1}-\rho|^3\right| \calG_{k-1}\right],\nonumber\\
	&\leq \frac1{|\T_n|^{3/2}} \sum_{k=1}^{t_n}\left(\E\left[\left.|V_{2k}|^3|V_{2k+1}|^3 \right| \calG_{k-1}\right] + 3 |\rho| \E\left[\left.V_{2k}^2V_{2k+1}^2 \right| \calG_{k-1}\right]\right.\label{majphi}\\
&\hspace{170pt} +\left.3\rho^2\E\left[\left.|V_{2k}||V_{2k+1}| \right| \calG_{k-1}\right]+|\rho|^3\right).\nonumber
\end{align}
It follows from Cauchy-Schwarz inequality together with the previous calculations that it exists two constants $\alpha,\beta>0$ such that
$$\E\left[\left.|V_{2k}||V_{2k+1}| \right| \calG_{k-1}\right] \leq \alpha c_k \hspace{20pt} \text{a.s.}$$
and
$$\E\left[\left.|V_{2k}|^3|V_{2k+1}|^3 \right| \calG_{k-1}\right] \leq \beta c_k^3 \hspace{20pt} \text{a.s.}$$
In addition, we already saw from \eqref{EV2V2} that for some constant $\gamma>0$
$$\E\left[\left.V_{2k}^2V_{2k+1}^2 \right| \calG_{k-1}\right] \leq \gamma c_k^2 \hspace{20pt} \text{a.s.}$$
Consequently, we obtain from \eqref{majphi} that for some constant $\delta>0$
$$\phi_n \leq \frac\delta{|\T_n|^{3/2}} \sum_{k=1}^{t_n} c_k^3 \hspace{20pt} \text{a.s.}$$
which, via Lemma \eqref{LFGN}, leads to
$$\lim_{n\to\infty} \phi_n =0 \hspace{20pt} \text{a.s.}$$
Hence, we can conclude that
$$H_{t_n}^{(n)} \liml \calN(0,\sigma_\rho^2).$$
In other words
$$\sqrt{|\T_{n-1}|} (\rho_{n} -\rho) \liml \calN(0,\sigma_\rho^2).$$
Finally, we find via \eqref{vitesserho} that
$$\sqrt{|\T_{n-1}|} (\wh\rho_{n} -\rho) \liml \calN(0,\sigma_\rho^2)$$
which achieves the proof of Theorem \ref{TCL}. \hfill $\square$\\

{\bf Acknowledgement.} I would like to thank Bernard Bercu for his helpful suggestions and for thorough readings of the paper.
\nocite{*}
\bibliographystyle{acm}
\bibliography{BINAR}

\end{document}